\documentclass{amsart}%
\usepackage{amsmath}
\usepackage{amsfonts}
\usepackage{amssymb}
\usepackage{graphicx}%
\usepackage{color}
\usepackage{amsfonts}%
\usepackage{amsmath}%

\usepackage{amscd}
\usepackage{amsthm}
\usepackage{epsfig}
\usepackage{amstext}
\usepackage[all]{xy}
\usepackage{graphicx}%
\usepackage{amssymb}%
\usepackage{graphicx}

\usepackage{makecell}

\providecommand{\U}[1]{\protect\rule{.1in}{.1in}}
\newtheorem{theorem}{Theorem}[section]

\newtheorem{corollary}[theorem]{Corollary}

\newtheorem{definition}[theorem]{Definition}

\newtheorem{lemma}[theorem]{Lemma}

\newtheorem{proposition}[theorem]{Proposition}
\newtheorem{remark}[theorem]{Remark}

\begin{document}
\title{Central amalgamation of groups and the RFD property}

\date{\today}

\author{Tatiana Shulman}

\address{Tatiana Shulman
\newline Institute of Mathematics of the Polish Academy of Sciences, Poland.}
\email{tshulman@impan.pl}



\begin{abstract} It is an old and challenging topic to investigate for which discrete groups $G$ the full group C*-algebra $C^*(G)$ is residually finite-dimensional (RFD). In particular not much is known about how the RFD property behaves under fundamental constructions, such as amalgamated free products and HNN-extensions.
In [CS19] it was proved that central amalgamated free products of virtually abelian groups are RFD. In this paper we prove that this holds much beyond this case.  Our method is based on showing  a   certain approximation property for characters induced from central subgroups. In particular it allows us to prove that free products of polycyclic-by-finite groups amalgamated over finitely generated central subgroups are RFD.

On the other hand we prove that the class of RFD C*-algebras (and groups) is not closed under central amalgamated free products. Namely  we give an example of RFD groups (in fact finitely generated amenable RF groups) whose central amalgamated free product is not RFD, moreover it is  not even maximally almost periodic. This answers a question of Khan and Morris [KM82].

\end{abstract}

\maketitle

\section{Introduction}

Let $G$ be a discrete group. When can we say that $G$ has many finite-dimensional representations? One way is to require that

\medskip

{\it finite-dimensional representations separate points of $G$}.

\medskip

\noindent In this case $G$ is called {\it maximally almost periodic} (MAP) which for finitely generated groups
is equivalent to being {\it residually finite} (RF).
Another way is to require for the full group C*-algebra $C^*(G)$ that

\medskip

{\it finite-dimensional representations separate points of $C^*(G)$}.

\medskip

\noindent Groups with this property are called {\it residually finite-dimensional} (RFD).

These two options are not the same. While RFD is clearly not weaker than MAP (since $G$ embeds into $C^*(G)$), they are not equivalent. In \cite{Bekka} Bekka constructed the first examples of groups   which are MAP but not RFD and later proved that $SL_3(\mathbb Z)$ is also such an example (\cite{Bekka2}). On the other hand for amenable groups these two notions coincide (\cite{BekkaLouvet}).  For some MAP groups proving/disproving RFD  is extremely hard -- as proved by Kirchberg (\cite{Kirchberg}, see also [\cite{CourtneySherman}, Th. 6.7]) the famous Connes Embedding Problem is equivalent to the question of whether or not $F_2\times F_2$ is RFD (while it is well known to be MAP).

Even besides the Connes Embedding Problem, exploring the RFD property for various classes of groups is a topic of much attention as it is important for our understanding of unitary
duals of groups (see e.g. \cite{LZuk}, \cite{LS}, \cite{Valette}). Moreover finding new examples of RFD groups has also become of interest  due to its relevance with problems of finding decidability algorithms for groups \cite{FNT}.

So far not many non-amenable examples of RFD groups are known (some examples can be found in \cite{ChoiFree},  \cite{LS},  \cite{DonTanya2}), so that amenable MAP groups remain the "main source" of groups with this property. For finding new non-amenable RFD groups it is therefore important to study how the RFD property behaves under group-theoretical constructions, such as amalgamated free products and HNN-extensions.

There are many examples of general  RFD  (not necessarily group) C*-algebras. Moreover some important problems about C*-algebras can be reduced to the case of RFD C*-algebras (see e.g. \cite{Dadarlat}). In this sense we can say that RFD C*-algebras form a rich class.  However when it comes to proving permanence properties for this class, it quickly becomes challenging, which  for instance can be seen in the study of how the RFD property behaves under amalgamated free products of C*-algebras.   In a way this is parallel to the study of how the RF property of finitely generated groups behaves under amalgamation of groups, which was a topic of great attention of group theorists for decades (see the survey \cite{survey}).
 It is interesting  to compare  how developments have proceeded so far on both sides.
In group theory the first result of this kind is the residual finiteness (RF) of finitely generated free groups proved in 1930 by Levi \cite{Levi}. Correspondingly the RFD property for free groups was proved by Choi in \cite{ChoiFree}. In 50s Gruenberg proved  that a free
  product of RF groups is RF \cite{Gru}.     On the C*-algebraic side, a question of whether free products of RFD C*-algebras
   (and hence groups) is RFD was open for long time, with
 only partial results known, until in 2008 it was solved in the positive by Exel and Loring \cite{ExelLoring}. Further, Baumslag
 showed that free products of RF groups amalgamated over a finite subgroup are RF \cite{Baumslag}.
 The corresponding result for RFD  C*-algebras fails.
 Nonetheless, necessary and sufficient conditions have been given for when amalgamated products of two separable RFD C*-algebras
 amalgamated over finite-dimensional C*-subalgebras are RFD, first for matrix algebras by Brown and Dykema  \cite{BD},
 then for finite-dimensional C*-algebras by Armstrong, Dykema, Exel, and Li \cite{ADEL}, and finally for all separable RFD
 C*-algebras by Li and Shen  \cite{LiShen}.

  Moving beyond finite amalgamating
subgroups, group theorists found a plethora of cases when amalgamated free products of RF groups are RF.
On the C*-side, the first and, until recently, only result involving an infinite-dimensional amalgamating
C*-subalgebra was a result of Korchagin stating that amalgamated free products of commutative C*-algebras are RFD \cite{Korchagin}.
On the other hand,  in \cite{Baumslag}, Baumslag proved that the amalgamated product of polycyclic groups over a  subgroup that is central in one of the groups must be RF, and in \cite{KhanMorris} Khan and Morris proved that the amalgamated product of two MAP topological groups over a common compact central subgroup is MAP.  These results, and especially Korchagin's result, indicated that the next natural class of amalgamated free products of C*-algebras and groups to
 test for being RFD is the central ones, that is the ones  amalgamated  over central C*-subalgebras (subgroups, respectively).

 In the recent work
 \cite{KristinTanya} K. Courtney and the author  proved that a central amalgamated free product of C*-algebras is RFD provided
 that the C*-algebras are {\it strongly RFD} meaning that all their quotients are RFD. (Warning: for group C*-algebras it is not the same as all quotients of the group being RFD).  For group C*-algebras strong RFD is a rare property (see a discussion of this property for groups in \cite{KristinTanya}). Therefore it is important to investigate how essential the requirement of strong RFD actually is. In particular whether it is in fact  necessary in the case when a central amalgamating subgroup is non-trivial.


 In this paper we prove that the strong RFD assumption is not necessary at all and that the RFD property is preserved by central amalgamated free products of a much bigger classes of groups.
For that we introduce a new technique which
   is based on obtaining two  results on full group C*-algebras, both of independent interest. The first one gives, under certain assumptions,  a simultaneous approximation of traces induced from a common central subgroup,
 by traces of finite-dimensional representations. Approximation of traces on (group) C*-algebras by the usual, that is matricial, traces  often plays an important role.  In particular approximation of traces by traces of finite-dimensional approximate representations is one of the key components in the classification programm (\cite{TWW}, \cite{Sch}, \cite{Gabe}).    Approximation of traces by traces of actual finite-dimensional
 representations is important at the study of the Hilbert-Schmidt stability of C*-algebras and groups (\cite{HadShTracial}, \cite{DonTanya2}). Thus the simultaneous approximation result which we obtain in this paper has potential applications at the study of the Hilbert-Schmidt stability of amalgamated free products. In this paper we apply it to study the RFD-property. The second auxiliary result we prove is a certain   generalization of the fact that for amenable groups their reduced and full C*-algebras coincide. This result will allow us to use arguments involving Voiculescu' Theorem to show that certain representations obtained by the GNS-construction separate points of central amalgamated free products of full C*-algebras of amenable groups.

Here is a particular case of our results on the RFD property of amalgamated free products:

\medskip

\noindent {\bf Theorem.}  {\it Let $G_1$ and $G_2$ be polycycic-by-finite groups and let $C$ be a central subgroup in both. Then the amalgamated free product $G_1\ast_C G_2$ is RFD.}

\medskip

The same technique allows us to prove similar results for central HNN-extensions, that is for HNN-extensions relative to the identity isomorphism of   an associated central subgroup, which are exactly those HNN-extensions in which the associated subgroup is central. In fact we give here a complete characterization of when a central HNN-extension of a finitely generated amenable group  relative to a finitely generated central subgroup is RFD.

\medskip

\noindent {\bf Theorem}: {\it Let $G$ be a finitely generated amenable group and let $C$ be its finitely generated central subgroup. Then the following are equivalent:

(i) the HNN-extension $\langle G, t\;|\; t^{-1}Ct = C, \; id\rangle$ is RF;

(ii) there exists a $C$-filtration of $G$;

(iii) the HNN-extension $\langle G, t\;|\; t^{-1}Ct = C, \; id\rangle$ is RFD.}

\medskip

Nevertheless, we show that quite surprisingly  in general  central amalgamated free
products of  RFD C*-algebras, even of RFD groups, need not be RFD.

\medskip

\noindent {\bf Theorem.} {\it There exists a finitely-generated amenable RFD group $G$ and its central subgroup $C$ such that $G\ast_C G$ is not MAP, and hence is not RF and is not RFD.}

\medskip

Recall that for locally compact groups maximal almost periodicity (MAP) is defined in the same way as for discrete ones with the only change that representations are assumed to be continuous (e.g. \cite{KhanMorris}, \cite{Spronk}). In [\cite{KhanMorris}, p. 428] there was posed a question of whether
 free products of locally compact MAP groups amalgamated over a common closed central subgroup are MAP. Since discrete groups are locally compact and all their subgroups are closed, our theorem above answers this question in the negative.

\medskip

The paper is organized as follows. Section 2 is preliminary and contains necessary information on C*-algebras and groups.
In section 3 we prove auxiliary results on existence of compatible filtrations  for RF groups with a common subgroup.    In section 4 we prove the mentioned above approximation property for traces, which will be (one of) the main ingredient(s) for obtaining positive results on RFD property of central amalgamated free products. Section 5 contains another mentioned above result which generalizes the fact that for amenable groups their reduced and full C*-algebras coincide.
Section 6 contains our main positive results on residual finite-dimensionality of central amalgamated free products. Section 7 is devoted to HNN-extensions. It contains the mentioned above characterization of when a central HNN-extension of a finitely generated amenable group  relative to a finitely generated central subgroup is RFD. Besides that it contains a result which relates the RFD property of a free product of a C*-algebra (or a group) with itself amalgamated over a general, not necessarily central, C*-subalgebra  (a subgroup, respectively) with the RFD property of the corresponding HNN-extension. In section 8 we prove our negative results for RFD property of central amalgamated free products.

\medskip

\textbf{Acknowledgements.}
The author is grateful to Alex Chirvasitu for useful comments on the first version of this paper.
The author was supported  by the grant H2020-MSCA-RISE-2015-691246-Quantum Dynamics and by the Polish National Science Center 
grant under the contract number 2019/34/E/ST1/00178.

\section{Preliminaries}
In this paper all groups are countable discrete.
By a representation of a group we always mean a {\it unitary} representation.

\subsection{MAP, RF, RFD}

 \begin{definition} A group $G$ is maximally periodic (MAP) if finite-dimensional representations of $G$ separate the points of $G$.
 \end{definition}

 \begin{definition} A group $G$ is residually finite (RF) if homomorphisms from $G$ to finite groups separate the points of $G$.
 \end{definition}

 \begin{definition} A C*-algebra $A$ is residually finite-dimensional (RFD) if finite-dimensional representations of $A$ separate the points of $A$.
 \end{definition}

 Recall that with a discrete group $G$ one can associate its {\it full C*-algebra}, $C^*(G)$. Namely let $\mathbb CG$ denote the group algebra of $G$. Let $A$ be a $C^*$-algebra and  let $\pi$ be a homomorphism from $G$ to the unitary group $U(A)$ of $A$.  Then $\pi$
induces a $\ast$-homomorphism $\pi : \mathbb CG \to A$. The full C*-algebra $C^*(G)$ is
the completion of $\mathbb C$G with respect to the norm
$$\|a\| := \sup\{\|\pi(a)\|\;|\; \pi: G \to U(A) \;\text{is a homomorphism}\}.$$
The C*-algebra $C^*(G)$ has the following universal property: Given a C*-algebra $A$ and a
unitary representation $\pi : G \to U(A)$, there exists a unique $\ast$-homomorphism
$\tilde \pi : C^*(G) \to A$  that satisfies $\tilde \pi(\delta(g)) = \pi(g)$, for every $g\in G$ (here $\delta: G \to \mathbb C G$ is
the canonical embedding).

\medskip

 \noindent One usually calls a group $G$ RFD if $C^*(G)$ is RFD. Clearly we have the following implications:
 $$G \;\text{ is RFD}\; \Rightarrow \; G \; \text{ is MAP}, \;\; G\; \text{ is RF}\; \Rightarrow \; G \; \text{ is MAP}.$$
  The opposite implications are not true in general.



\noindent The following proposition is folklore.

\begin{proposition}\label{MAPhenceRF} A finitely generated MAP group is RF.
\end{proposition}
\begin{proof} Let $e\neq g\in G$. Since $G$ is MAP, there is a finite-dimensional representation $\pi$ of $G$ such that $\pi(g)\neq 1$. Since $\pi(G)$ is a finitely generated linear group, by Malcev Theorem (\cite{Malcev}, Th. VII) it is RF. Therefore there exists a homomorphism $f$ from $\pi(G)$ to a finite group such that $f(\pi(g))\neq e$. Thus the homomorphism $f\circ \pi$ to a finite group is not trivial on $g$.
\end{proof}

\subsection{Amalgamated free products and  HNN-extensions.}

Recall that if $\phi_A : C \to A, \phi_B : C \to B$ are unital injective $\ast$-homomorphisms of unital C*-algebras then their {\it amalgamated free product}
$A\ast_C B$  is a unital C*-algebra
with the following properties:

1)  There exist unital $\ast$-homomorphisms $i_A : A \to A\ast_C B$  and $i_B : B \to A\ast_C B$
 such that the
following diagram
\[
\xymatrix{C \ar[d]_{\phi_B} \ar[r]^{\phi_A} & A \ar[d]^{i_A}\\
 B  \ar[r]_{i_B} & A\ast_C B}
 \]
 is commutative.

 2) For any unital C*-algebra $D$ and any commutative diagram \[
\xymatrix{C \ar[d]_{\phi_B} \ar[r]^{\phi_A} & A \ar[d]^{\pi_A}\\
 B  \ar[r]_{\pi_B} & D}
 \]
 there is a unique $\ast$-homomorphism $\sigma : A\ast_C B \to D$ such that $\sigma \circ i_A = \pi_A$  and
$\sigma \circ i_B = \pi_B$.

Such C*-algebra exists and is unique up to isomorphism. It is easy to show that $i_A$ and $i_B$ are in fact injective, so when it does not lead to a confusion we assume $A, B \subset  A\ast_C B$.

\medskip

Let $A$ be a unital $C^*$-algebra, $B$ and $C$ its C*-subalgebras, and $\phi: B\to C$ an isomorphism. The corresponding {\it HNN-extension}
$\langle A, t\;|\; t^{-1}Bt = C, \; \phi\rangle$ is a unital C*-algebra with the following properties:

1) There exists a unital $\ast$-homomorphism $i_A: A \to \langle A, t\;|\; t^{-1}Bt = C, \; \phi\rangle$ and a unitary $t\in \langle A, t\;|\; t^{-1}Bt = C, \; \phi\rangle$ such that
$$t^{-1}i_A(b)t= i_A(\phi(b)),$$ for any $b\in B$.

2) For any unital C*-algebra $D$ and any unitary $u\in D$ and unital $\ast$-homomorphism $\pi: A \to D$ such that $u^{-1}\pi(b)u = \pi(\phi(b))$, for any $b\in B$, there is a unique $\ast$-homomorphism $\sigma: \langle A, t\;|\; t^{-1}Bt = C, \; \phi\rangle \to D$ such that
$\sigma\circ i_A = \pi, \; \sigma(t) = u.$

Such C*-algebra exists and is unique up to isomorphism.

More information on amalgamated free products and HNN-extensions of C*-algebras can be found in e.g.  \cite[section 2.3]{Pedersen} and \cite{Ueda}.


\medskip

In the case of full group C*-algebras these definitions agree with the group-theoretical definitions of amalgamated free products and HNN-extensions. Namely  (see e.g. \cite[Lemma 3.1 and Lemma 3.4]{SorenTanyaAdam})
$$C^*(G_1\ast_{H} G_2) = C^*(G_1)\ast_{C^*(H)} C^*(G_2)$$ and
$$C^*\langle G, t\;|\; t^{-1}Ht = K, \; \phi\rangle = \langle C^*(G), t\;|\; t^{-1}C^*(H)t = C^*(K), \; \phi\rangle.$$

Sometimes, when it is important to emphasize an isomorphism between the amalgamating subgroups, we will write $G_1\ast_{H_1\cong^{\phi} H_2} G_2$   rather than $G_1 \ast_H G_2$.

\subsection{Positive-definite functions and characters, states and traces.}

\begin{definition} A function $\phi: G \to \mathbb C$ is positive definite if the matrix $$[\phi(s^{-1}t)]_{s, t\in F} \in M_F(\mathbb C)$$ is positive for every finite set $F\subset G$.
\end{definition}

\begin{definition} A character is a positive definite function that is constant on conjugacy classes.
\end{definition}

 It is well-known (\cite{BrownOzawa}, p.47) that a positive definite function $\phi: G \to \mathbb C$
 extends to a positive linear functional on the full group C*-algebra $C^*(G)$ and this correspondence is 1-to-1. Similarly characters on $G$ are in 1-to-1 correspondence with traces on $C^*(G)$.
In this paper a trace on $C^*(G)$ and the corresponding character of $G$ (a positive linear functional on $C^*(G)$ and the corresponding positive definite function on $G$, respectively) will be denoted by the same letter.



\medskip

\noindent Suppose $\lambda$ is a positive-definite function on a subgroup $C$.
Let $\tilde\lambda(g) = \begin{cases} \lambda(g), g\in C \\ 0, g\in G\setminus C \end{cases}$.
The following proposition is probably well-known but we include its proof for readers convenience.

\begin{proposition} Let $G$ be a group and let $C$ be its subgroup. If $\lambda$ is a positive-definite function on $C$, then  $\tilde \lambda$ is a positive definite function on $G$. If $C$ is central, then $\tilde \lambda$ is a character of $G$.
\end{proposition}
\begin{proof}
Let $g_1C, g_2C, \ldots$ be the cosets for $C$. Then $G = g_1C\sqcup g_2C \sqcup \ldots$. Then  $g^{-1}h\in C$ is and only if $g, h$ belong to the same coset. Let $t_g\in \mathbb C,$ for $g\in G$. We have
\begin{multline*} \sum_{g, h}\tilde\lambda(g^{-1}h) t_g\overline{t_h} = \sum_{g^{-1} h\in C}\lambda(g^{-1}h) t_g\overline{t_h}\\
= \sum_i\sum_{g, h\in g_iC} \lambda(g^{-1}h) t_g\overline{t_h} = \sum_i\sum_{c, c'\in C} \lambda(c'^{-1}c)t_{g_ic'}\overline{t_{g_ic}} \ge 0.\end{multline*}
This means that $\tilde\lambda$ is positive definite.

Now assume that $C$ is central.
       Let $g, h\in G$. If $h\in C$, then $ghg^{-1}=h$ and $\tilde\lambda(ghg^{-1}) = \tilde\lambda(h).$
       If $h\notin C$, then $ghg^{-1}\notin gCg^{-1} = C$ and we have $\tilde\lambda(ghg^{-1}) = 0 =\tilde\lambda(h).$
  Thus $\tilde\lambda$ is constant on conjugacy classes.
\end{proof}

\noindent Throughout this paper a homomorphism from a group $G$ to the unit circle $\{|z|=1\}$ will be called a one-dimensional representation of $G$.


\subsection{GNS-construction.}

 Let $\tau$ be a positive linear functional on a C*-algebra $A$. Let
 $$N_{\tau} = \{a\in A\;|\; \tau(a^*a)=0\}.$$
 It is easy to check that $N_{\tau}$ is a closed left ideal of $A$ and that the map
 $$(a+N_{\tau}, b+N_{\tau}) \mapsto \tau(b^*a)$$ is a well-defined inner product on $A/N_{\tau}$.
 Let $H_{\tau}$ be the Hilbert completion of $A/N_{\tau}$.

For any $a\in A$, we define an operator
$$b+N_{\tau} \mapsto ab + N_{\tau}$$ on $A/N_{\tau}$. This operator has a unique extension to a bounded operator
$\Lambda_{\tau}(a)$ on $H_{\tau}$. The map
$$\Lambda_{\tau}: A \to B(H_{\tau}), \; a\mapsto \Lambda_{\tau}(a),$$ is a $\ast$-homomorphism called the {\it Gelfand-Naimark-Segal representation} (or $GNS$-representation) of $A$ associated with $\tau$.


The following lemma follows easily from the GNS-construction.

\begin{lemma}\label{basis} Let $G$ be a group and $C$ be its subgroup  and let $\lambda$ be a positive-definite function on  $C$. Let $g_1, g_2, \ldots$ be representatives of left cosets of $C$. Then  $\{g_i+N_{\tilde\lambda}\}$, $i\in \mathbb N$, is an orthonormal basis in the space $H_{\tilde\lambda}$ of the GNS-representation $\Lambda_{\tilde\lambda}$ of $C^*(G)$.
\end{lemma}

\subsection{Voiculescu's Theorem.}

Let $H$ be a separable Hilbert space. Two representations $\pi$ and $\rho$ of a
C*-algebra $A$ on $H$ are {\it approximately unitarily equivalent}, denoted
by $\pi \sim_a \rho$, if there exists a sequence $U_n$, $n\in \mathbb N$, of unitary operators on $H$ such that
$$\lim_{n\to \infty} \|\pi(a) - U_n^*\rho(a) U_n\| =0,$$ for any $a\in A$.

For $T \in \mathcal B(H)$, we let $rank (T)\in \mathbb N\bigcup\{\infty\}$ denote the Hilbert-space dimension of the closure of the range
$Ran(T)$ of $T$.

In this paper we will use Voiculescu's Theorem on approximate unitary equivalence. Here is its reformulation due to Hadwin \cite{ReformulationVoiculescu}.

\begin{theorem} Let $A$ be a separable  C*-algebra, $H$ a separable Hilbert space
and $\pi, \rho: A \to B (H)$ non-degenerate representations. The following are equivalent:

(i) $\pi \sim_a \rho$,

(ii) $rank(\pi(a)) = rank(\rho(a)),$ for any $a\in A$.
\end{theorem}
Throughout this paper we will apply Voiculescu's Theorem to unital representations which are automatically non-degenerate.

\medskip

 If $H$ is a Hilbert space and $H_0\subset H$ is a closed subspace, then we consider a representation $\pi$ of a $C^*$-algebra $A$ on $H_0$ as a (degenerate) representation of $A$ on $H$. Throughout this paper we will use the following notation: If $\pi$ is a representation of $A$ on a Hilbert space $H$, and $\pi_n$'s are representations of $A$ on $H_n\subseteq H$, then we write
 $$\pi = SOT\!-\!\lim \pi_n$$ meaning that for each $a\in A$, $\pi_n(a)\to \pi(a)$ in the strong operator topology.

\section{Auxilliary results on RF groups}
For a group $G$  its center will be denoted by $Z(G)$. We  use notation $N\triangleleft_f G$ for  a normal subgroup  of finite index.

\begin{definition}\label{filtration}(Baumslag \cite{Baumslag}) Let $A$ be a group. A family of its normal subgroups $A_{\lambda}$ of finite index is called a {\it filtration} if
 $$\bigcap A_{\lambda} = \{e\}.$$  Let $H$ be a subgroup of $A$. A filtration $\{A_{\lambda}\}_{\lambda\in \Lambda}$ is a called an {\it $H$-filtration} if
 $$\bigcap HA_{\lambda} = H.$$ Let $B$ be a second group
with a distinguished subgroup $K$ which is
isomorphic to $H$ under a given isomorphism $\phi$, and let $\{B_{\lambda}\}_{\lambda\in \Lambda}$ be a $K$-filtration of $B$. Then we say that $\{A_{\lambda}\}_{\lambda\in \Lambda}$ and $\{B_{\lambda}\}_{\lambda\in \Lambda}$ are  {\it $(H, K, \phi)$-compatible} if  the mapping
$hA_{\lambda} \mapsto \phi(h)B_{\lambda}$ is well-defined and  an isomorphism between $HA_{\lambda}/A_{\lambda}$ and $KB_{\lambda}/B_{\lambda}$, for each $\lambda\in \Lambda$.
\end{definition}

We will need an easy reformulation.

\begin{proposition}\label{ReformulationFiltr} The following are equivalent:

(i) There exist $(H, K, \phi)$-compatible $H$-filtration of $A$ and $K$-filtration of $B$;

(ii) For any $a_1, \ldots, a_n\notin H$, $b_1, \ldots, b_m\notin K$, $h_1, \ldots, h_l\in H$ and $k_1, \ldots, k_s\in K$ there exist homomorphisms $$f^A: A \to G^A, \; f^B: B \to G^B$$ onto finite groups $G^A, G^B$,  such that $f^A(H)$ and $f^B(K)$ are isomorphic via the mapping $f^A(h) \mapsto f^B(\phi(h))$, $h\in H$, and

1) $f^A(a_i)\notin f^A(H)$, for any $i=1, \ldots, n$,

2) $f^B(b_i)\notin f^B(K)$, for any $i=1, \ldots, m$,

3)$f^A(h_i)\neq f^A(h_j)$, for any $i\neq j$, $i, j =1, \ldots, l$,

4)$f^B(k_i)\neq f^B(k_j)$, for any $i\neq j$, $i, j =1, \ldots, s$.
\end{proposition}
\begin{proof} (i) $\Rightarrow$ (ii): Let $a_1, \ldots, a_n\notin H$, $b_1, \ldots, b_m\notin K$, $h_1, \ldots, h_l\in H$ and $k_1, \ldots, k_s\in K$. It follows from Definition \ref{filtration} that for any $i=1, \ldots, n$ we can find  $\lambda_i$ such that $a_i\notin HA_{\lambda_i}$, for any $i=1, \ldots, m$ we can find $\lambda_i'$ such that $b_i\notin KB_{\lambda_i'}$, for any $i\neq j$, $i, j =1, \ldots, l$ we can find $\lambda_{i,j}$
such that $h_ih_j^{-1}\notin A_{\lambda_{i,j}}$, and for any $i\neq j$, $i, j =1, \ldots, s$ we can find $\lambda_{i, j}'$ such that $k_ik_j^{-1}\notin B_{\lambda_{i, j}'}.$
 Since
$$\left(\bigcap_{i=1}^n A_{\lambda_i}\right)\bigcap\left(\bigcap_{ i\neq j = 1}^l A_{\lambda_{i, j}}\right)\; \text{and} \;  \left(\bigcap_{i=1}^m B_{\lambda_i'}\right)\bigcap\left(\bigcap_{i\neq j=1}^s B_{\lambda_{i, j}'}\right)$$ are normal subgroups of finite index,
the quotient maps $$f^A: A \to A/\left(\bigcap_{i=1}^nA_{\lambda_i}\right)\bigcap\left(\bigcap_{i\neq j=1}^l A_{\lambda_{i, j}}\right)$$ and $$f^B: B \to B/\left(\bigcap_{i=1}^mB_{\lambda'_i}\right)\bigcap\left(\bigcap_{i\neq j=1}^s B_{\lambda'_{i, j}}\right)$$ are homomorphisms to finite groups. It is easy to see that they satisfy all the conditions of (ii).

(ii) $\Rightarrow$ (i): Let $$\Lambda = \{a_1, \ldots, a_n\notin H, b_1, \ldots, b_m\notin K, h_1, \ldots, h_l\in H, k_1, \ldots, k_s\in K \;|\; n, m, l, s\in \mathbb N\}. $$ For $\lambda\in \Lambda$ we find $f^A, f^B$ as in (ii) and define $$A_{\lambda} = \ker f^A, \; B_{\lambda} = \ker f^B.$$
It is easy to check that we obtained an $H$-filtration of $A$ and a $K$-filtration of $B$ that are $(H, K, \phi)$-compatible.
\end{proof}

\begin{proposition}\label{QuotientRF0} Let $G$ be a group and $H$ its normal subgroup. The following are equivalent:

 (i) $\bigcap_{N \triangleleft_f G} HN = H$;

 (ii) $G/H$ is RF.
\end{proposition}
\begin{proof}
(i)$\Rightarrow$(ii): Suppose $g\notin H$. By (i) there exists $N \triangleleft_f G$ such that $g\notin HN$. It follows that for the homomorphism $\pi: G \to G/N$ one has $\pi(g)\notin \pi(H)$. Now for the homomorphism $\rho: G/H\to \pi(G)/\pi(H)$ to the finite group $pi(G)/\pi(H)$ one obtains $\rho(gH) = \pi(g)\pi(H) \neq e$.

(ii)$\Rightarrow$(i): Clearly $\bigcap_{N \triangleleft_f G} HN \supseteq H$. Suppose $g\notin H$. Since $G/H$ is RF, there is a homomorphism $\pi: G/H \to A$, where $A$ is a finite group, such that $\pi(gH)\neq e$. Let $q: G \to G/H$ be the canonical surjection and let $N = \ker  \pi\circ q.$ Then $N \triangleleft_f G$, $g\notin N$, and $H\subseteq N$ whence $HN = N$. Hence $g\notin HN$ and therefore $\bigcap_{N \triangleleft_f G} HN = H$.

\end{proof}

\begin{corollary}\label{QuotientRF} Let $G$ be an RF group and $H$ its normal subgroup such that  $G/H$ is RF. Then  there exists an $H$-filtration of $G$.
\end{corollary}

\begin{corollary}\label{QuotientCenter} Let $G$ be an RF group and let $Z(G)$ be its center. Then
$$\bigcap_{N \triangleleft_f G} N Z(G) = Z(G)$$ and $G/Z(G)$ is RF.
\end{corollary}
\begin{proof} Let $x\notin Z(G)$. Then there is $y\in G$ such that $xyx^{-1}y^{-1}\neq e.$
As $G$ is RF, there is a homomorphism $\pi$ to a finite group such that $\pi(x) \pi(y)\pi(x)^{-1}\pi(y)^{-1}\neq e$.
Hence $\pi(x)\notin \pi(Z(G))$ and hence $x\notin  Z(G) Ker \pi$. Hence $x\notin \bigcap_{N \triangleleft_f G} Z(G) N$ and
we showed that $\bigcap_{N \triangleleft_f G} N Z(G) = Z(G)$. The last statement follows from Proposition \ref{QuotientRF0}.
\end{proof}

Recall that a subgroup $H$ of a group $G$ is {\it profinitely closed}  if it is the intersection of all subgroups of finite index in $G$ that contain $H$. A group $G$ is {\it extended residually finite} (ERF) if every subgroup of $G$ is profinitely closed.

\begin{lemma}\label{ProfinitelyClosed} Let $H$ be a normal subgroup of $G$. Then $H$ is profinitely closed if and only if it is the intersection of all the normal subgroups of finite index in $G$ that contain $H$.
\end{lemma}
\begin{proof} We need to show that if $g\notin M$, for some subgroup $M$ of finite index in $G$, then $g\notin N$, for some normal subgroup $N$ of finite index in $G$. Since $g\notin M$, we have $g\notin \bigcap_{g_0\in G} g_0Mg_0^{-1}.$ Let $N = \bigcap_{g_0\in G} g_0Mg_0^{-1}$. Then $N$ is normal. Since $H$ is normal, we have $g_0Mg_0^{-1} \supseteq g_0Hg_0^{-1} = H$, for any $g_0\in G$. Hence $N\supseteq H$. Let $g_1, \ldots, g_N$ be representatives of left cosets of $M$. Then $N = \bigcap_{g_0\in G} g_0Mg_0^{-1} = \bigcap_{i=1}^{N} g_iMg_i^{-1}$ is the intersection of finitely many subgroups of finite index and hence is of finite index.
\end{proof}

\begin{theorem}\label{WhenCompFiltrExists} Let $A$ be a group and $H$ be its   subgroup  such that there exists an $H$-filtration of $A$. Let $B$ be an ERF group, $K$  a central subgroup of $B$,  $\phi: H \to K$  an isomorphism.
Then there exist $(H, K, \phi)$-compatible $H$-filtration of $A$ and $K$-filtration of  $B$.
\end{theorem}
\begin{proof} Let $a_1, a_2, \ldots$ be all the elements in $A\setminus H$, $h_1, h_2, \ldots$ -- all the elements in $H\setminus \{e\}$,
$b_1, b_2, \ldots$ -- all the elements in $B\setminus K$, $k_1=\phi(h_1), k_2=\phi(h_2), \ldots$ -- all the elements in $K\setminus \{e\}$.
Fix $n\in \mathbb N$. Since there exists an $H$-filtration of $A$, we can find  $A_n \triangleleft_f A$ such that
\begin{equation}\label{When**'} a_n\notin A_nH, \;  h_n\notin A_n. \end{equation}
Let
$$L_n = \phi\left(H\bigcap A_n\right).$$
It follows that
\begin{equation}\label{k=phi(h)}k_n\notin L_n.\end{equation}
Since  $L_n$ has a finite index in $ \phi(H) = K$, we can denote by $k^{(1)}, \ldots, k^{(m_n)}$ representatives of cosets $K/L_n$, here $m_n<\infty$.
Let
$$E = \{N\;| \;N\triangleleft_f B, \; N\supseteq L_n\}.$$
Since $K$ is central, $L_n$ is normal in $B$ and by Lemma \ref{ProfinitelyClosed}
\begin{equation}\label{When1} \bigcap_{N\in E} N = L_n. \end{equation}
Let
$$F = \{M\; |\; M\triangleleft_f B, \; M\supseteq K\}.$$
Then $F\subseteq E.$ By Lemma \ref{ProfinitelyClosed} $K= \bigcap_{M\in F} M.$ Since for any $M\in F$ we have $M=KM$,
$$K  =\bigcap_{M\in F} M = \bigcap_{M\in F} KM \supseteq \bigcap_{N\in E} KN.$$ Clearly $K \subseteq \bigcap_{N\in E} KN$ and therefore
\begin{equation}\label{When2} K = \bigcap_{N\in E} KN.\end{equation}
By (\ref{k=phi(h)}), (\ref{When1}) and (\ref{When2}) there is $B_n \triangleleft_f B$ such that
\begin{equation}\label{When3*} B_n\supseteq L_n, \; b_n\notin KB_n, k_n\notin B_n\end{equation}
 \begin{equation}\label{When4*} k^{(1)}, \ldots, k^{(m_n)}\notin B_n.\end{equation}
 By (\ref{When4*}), $B_n\bigcap K = L_n.$ (Indeed, if $B_n \bigcap K \supset L_n$, then there is $k\in B_n$ which can written as $k=k^{(i)}l$, for some $i\le m_n$ and some $l\in L_n$. Then $k^{(i)} = kl^{-1}\in B_n$ which contradicts to (\ref{When4*})). Thus we obtain
\begin{equation}\label{When3} B_n\bigcap K = \phi\left(A_n\bigcap H\right).\end{equation}
By (\ref{When**'}), $\{A_n\}_{n\in \mathbb N}$ is an $H$-filtration of $A$, by (\ref{When3*}), $\{B_n\}_{n\in \mathbb N}$ is a $K$-filtration of $B$, and it remains to show that they are $(H, K, \phi)$-compatible.  By (\ref{When3}), $\phi$ induces a well-defined injective homomorphism $ HA_n/A_n \to KB_n/B_n$. (Indeed for the homomorphism $hA_n \mapsto \phi(h)B_n$ to be  well-defined and injective one needs the condition $h\in A_n$ to imply $\phi(h)\in B_n$, which is guaranteed by (\ref{When3})). Its surjectivity is obvious.
\end{proof}

The following corollary provides conditions that guarantee the existence of compatible filtrations. In the case of both groups being polycyclic it was most probably known to Baumslag as one can guess from [\cite{Baumslag}, Th. 8]. But since the proof of this was not written out (and the statement itself was not formulated) we have to prove it here as we will use it in section 6.

\begin{corollary}\label{WhenCompFiltrExists2} Suppose $A$ is RF,  $H$ is a subgroup of $A$, $K$ is a central subgroup of $B$, $\phi: H \to K$ is an isomorphism.
Then there exist an $H$-filtration of $A$ and a $K$-filtration of $B$ which are $(H, K, \phi)$-compatible, if any of the following conditions holds:

(i) $H = Z(A)$, $B$ is polycyclic-by-finite;

(ii) $H$ is normal, $A/H$ is RF, $B$ is polycyclic-by-finite;

(iii) $A, B$ are polycyclic-by-finite.
\end{corollary}
\begin{proof} It is a well-known theorem of Malcev that polycyclic-by-finite groups are ERF (see \cite{Lennox}, p.18). Therefore the statements follow now from Theorem \ref{WhenCompFiltrExists}, Corollary \ref{QuotientRF} and Corollary \ref{QuotientCenter}.
\end{proof}

As the existence of compatible filtrations implies that the corresponding amalgamated product is RF ([\cite{Baumslag}, Prop. 2]), we obtain the following corollary.

\begin{corollary} Suppose that $A$ is RF, $K$ is a subgroup in both $A$ and $B$ and is central in $B$. Then $A\ast_K B$ is RF if any of the following conditions holds:

(i) $K = Z(A)$, $B$ is polycyclic-by-finite;

(ii) $K$ is normal in $A$ and $A/K$ is RF, $B$ is polycyclic-by-finite;

(iii) (Baumslag [\cite{Baumslag}, Th. 8], Wehrfritz [\cite{Wehrfritz}, Th. 6]) $A$, $B$ are polycyclic-by-finite.
\end{corollary}

\section{Approximation of characters induced from the center}

In \cite{DonTanya2} D. Hadwin and the author proved the following approximation property for characters induced from the center of an RF group: If $\lambda$ is a character of the center (or of any central subgroup) $H$ and $\tilde\lambda =  \begin{cases} \lambda(a), a\in H \\ 0, a\in A\setminus H \end{cases}$, then there exist finite-dimensional representations $\pi_n$ of $G$ such that $$\tilde\lambda(g) = \lim_{n\to \infty} tr\; \pi_n(g),$$ for any element $g$ of the group. Here we will prove a refined version of this approximation property in the case when $\lambda$ is a 1-dimensional representation of a central subgroup. Namely for two (or finitely many) groups  with a common central subgroup $H$  we will construct approximations as above so that  the corresponding representations would coincide on $H$ and take scalar values on it. However for that we will need a stronger assumption than RF, namely  we will assume the existence of compatible filtrations.
 This refined approximation property will be crucial for the proof of results of section 6.

\begin{lemma}\label{LemmaRefinedTracesAppr} Let $A$ and $B$ be groups, $H$ and $K$ be their finitely generated central subgroups respectively, and $\phi: H \to K$ an isomorphism such that there exist $(H, K, \phi)$-compatible filtrations of $A$ and $B$. Let $\lambda$ be a 1-dimensional representation of $H$. Let $h_1, \ldots, h_N$ be generators of $H$,  $a_1, \ldots, a_n\in A \setminus H$, $b_1, \ldots, b_m\in B\setminus K$, $\epsilon >0$. Then there exist finite groups $G^A$, $G^B$, surjective homomorphisms $f^A: A \to G^A$, $f^B: B \to G^B$, and 1-dimensional representations $\chi^A$
and $\chi^B$ of $f^A(H)$ and $f^B(K)$ respectively, such that  $$\chi^A(f^A(h)) = \chi^B(f^B(\phi(h))),$$ for any $h\in H$,
$$|\chi^A(f^A(h_i)) - \lambda(h_i)| \le \epsilon,$$ for any $i=1, \ldots, N$, and
$$f^A(a_i)\notin f^A(H), \; f^B(b_j) \notin f^B(K),$$ for any $i=1, \ldots, n$, $j=1, \ldots, m$.
\end{lemma}
\begin{proof}
Since $H$ is a finitely generated abelian group, it can be written  as $$H = \mathbb Z^{s}\times \Gamma,$$ where $s\in \mathbb N$ and $\Gamma$ is a finite abelian group.
So we can write $h_j = (n_1^j, n_2^j, \ldots, n_s^j, t^j)$ with $n_i^j\in \mathbb Z$, $t^j\in \Gamma$, $j\le N$. Let $\mathbb Z^{(i)}$ denote the i-th copy of $\mathbb Z$ in $H$.
For each $i\le s$ there is $\theta_i$ such that
\begin{equation}\label{RF2}\lambda |_{\mathbb Z^{(i)}}(n) = e^{2\pi i n \theta_i}.\end{equation}
Let $$L^{(i)} = \max_{j\le N} |n_i^{j}|,$$ $i = 1, \ldots, s$. For each $i\le s$ there exists  $k_{0, i}$ such that for any $k\ge k_{0, i}$, the $k$-th roots of unity form an $\frac{\epsilon}{s(L^{(i)} +1)}$-net in the unit circle.

By Lemma \ref{ReformulationFiltr}  there  exist  finite groups $G^A$ and $G^B$ and  surjective homomorphisms $f^A: A\to G^A$ and $f^B: B \to G^B$ such that  \begin{equation}\label{LemmaRefinedTracesAppr1}f^A(n_1, \ldots, n_s, t) \neq f^A(n_1', \ldots, n_s', t'),\end{equation} when $t\in \Gamma$, $n_i, n_i' \le k_{0, i}$, and the tuples $(n_1, \ldots, n_s, t)$ and $(n_1', \ldots, n_s', t')$ do not coincide;
\begin{equation}\label{LemmaRefinedTracesAppr2} f^A(a_i)\notin f^A(H), \; f^B(b_j) \notin f^B(K),
\end{equation}
for any $i=1, \ldots, n$, $j=1, \ldots, m$; and
\begin{equation}\label{LemmaRefinedTracesAppr3} f^A(h) \mapsto f^B(\phi(h))\end{equation} is an isomorphism between $f^A(H)$ and $f^B(K)$.

  As $f^A(H) \cong (\prod_{i\le s}f^A(\mathbb Z^{(i)})) \times f^A(\Gamma)$, we can write $$f^A(H) = \mathbb Z_{k_1}\times \ldots \times \mathbb Z_{k_s} \times \tilde\Gamma, $$ for some $k_1, \ldots, k_s \in \mathbb N$ and some finite abelian group $\tilde\Gamma$. It follows from (\ref{LemmaRefinedTracesAppr1}) that $k_i \ge k_{0, i}$ and that $|\tilde \Gamma| \ge |\Gamma|$. Since $\tilde \Gamma$ is a homomorphic image of $\Gamma$, the latter implies that $\tilde \Gamma \cong \Gamma$. The first inequality, $k_i \ge k_{0, i}$, implies that there is $l_i< k_i$ such that
\begin{equation}\label{RF3} |e^{2\pi i l_i/k_i} - e^{2 \pi i \theta_i}| \le \frac{\epsilon}{s(L^{(i)}+1)}.\end{equation}

 Define  a 1-dimensional representation $\chi_i$ of $\mathbb Z_{k_i}$ by $$\chi_i(m \bmod k_i) = e^{2\pi im l_i/k_i},$$  $m \in \mathbb Z$. Using (\ref{RF3}),  for any $m$  we easily  obtain by induction that $$|\chi_i(m \bmod k_i) - \lambda |_{\mathbb Z^{(i)} } (m) | =
|e^{2\pi iml_i/k_i} - e^{2\pi i m \theta_i}|  \le \frac{\epsilon (m+1) }{s(L^{(i)}+1)}.$$
In particular for any $m\le L^{(i)}$ we obtain \begin{equation}\label{RF4}|\chi_i(m \bmod k_i) - \lambda |_{\mathbb Z^{(i)} } (m) | \le \frac{\epsilon  }{s}.\end{equation}

Define a 1-dimensional representation $ \chi^A$ of $f^A(H)$
by $$\chi^A(f^A(n_1, \ldots, n_s, t)) = \chi_1(n_1 \bmod k_1)\ldots\chi_s(n_s \bmod k_s)\lambda(t),$$ for all $n_i \in \mathbb Z, t\in \Gamma$. From (\ref{RF4}) we easily obtain that for any $n_i \le L^{(i)}$, $t\in \Gamma$, $$|\chi^A(f^A(n_1, \ldots, n_s, t)) - \lambda(n_1, \ldots, n_s, t)| \le \epsilon.$$ Hence \begin{equation}\label{LemmaRefinedTracesAppr4} |\chi^A(f^A(h_i)) - \lambda(h_i)| \le \epsilon,\end{equation} for $i = 1, \ldots, N$. Using (\ref{LemmaRefinedTracesAppr3}) we can define a 1-dimensional representation $\chi^B$ of $f^B(K)$ by
$$\chi^B(f^B(\phi(h))) = \chi^A(f^A(h)),$$ $h\in H$.
This, together with (\ref{LemmaRefinedTracesAppr2}) and (\ref{LemmaRefinedTracesAppr4}),  completes the proof.
\end{proof}

\begin{theorem}\label{TracesInducedFromCenter} Let $A$ and $B$ be groups, $H$ and $K$ be their finitely generated central subgroups respectively, and $\phi: H \to K$ an isomorphism such that there exist $(H, K, \phi)$-compatible filtrations of $A$ and $B$. Let $\lambda$ be a 1-dimensional representation of $H$,
$$\tilde\lambda^A (a)= \begin{cases} \lambda(a), a\in H \\ 0, a\in A\setminus H \end{cases},\;\; \tilde\lambda^B (b)= \begin{cases} \lambda(\phi^{-1}(b)), b\in K \\ 0, b\in B\setminus K \end{cases}.$$ Then there exist finite-dimensional representations $\pi_n^A, \pi_n^B$ of $A$ and $B$ respectively, such that for each $h\in H$, $\pi_n^A(h)$ and $\pi_n^B(\phi(h))$ are  scalar operators with the same scalars,  and
$$\tilde\lambda^A(a) = \lim tr \pi_n^A(a), \; \tilde\lambda^B(b) = \lim tr \pi_n^B(b),$$ for any $a\in A$, $b\in B$.
\end{theorem}
\begin{proof} It will be sufficient, given any $\tilde h_1, \ldots, \tilde h_M\in H$,  $a_1, \ldots, a_n\in A \setminus H$, $b_1, \ldots, b_m\in B\setminus K$, $\tilde\epsilon >0$, to find finite-dimensional representations $\pi^A$ and $\pi^B$ of $A$ and $B$ respectively such that:

1) $\pi^A(h)\;\text{and}\;\pi^B(\phi(h))$ are scalar operators with the same scalars, for each $ h\in H,$

2) $|tr \pi^A(\tilde h_i) - \lambda(\tilde h_i)|< \tilde \epsilon, \text{for any} \;i=1, \ldots, M,$

\noindent (and then automatically $|tr \pi^B(\phi(\tilde h_i)) - \lambda(\tilde h_i)|< \tilde \epsilon, \text{for any} \;i=1, \ldots, M$),
and

3) $|tr \pi^A(a_i)| < \tilde \epsilon, \; |tr \pi^B(b_j)|<\tilde\epsilon,$ for any $i=1, \ldots, n$, $j=1, \ldots, m$.

Let $h_1, \ldots, h_N$ be generators of $H$ and let $L$ be the maximal length of $\tilde h_i$'s, $i = 1, \ldots, M$, written as monomials of the generators $h_1, \ldots, h_N$. Let $$\epsilon = \frac{\tilde\epsilon}{L}.$$ We find $f^A, f^B, G^A, G^B, \chi^A, \chi^B$ as in Lemma \ref{LemmaRefinedTracesAppr}.
Let $Ind \chi^A$ be the representation of $G^A$ induced from $\chi^A$  and let $Ind \chi^B$ be the representation of $G^B$ induced from  $\chi^B$. Let
$$\pi^A = Ind \chi^A \circ f^A, \;\; \pi^B = Ind \chi^B \circ f^B.$$ Then
\begin{equation}\label{RefinedTracesAppr1} \pi^A(h) = \chi^A(f^A(h)) 1\end{equation}
 and by Lemma \ref{LemmaRefinedTracesAppr}
\begin{equation}\label{RefinedTracesAppr01} \pi^B(\phi(h) = \chi^B(f^B(\phi(h))) 1  = \chi^A(f^A(h)) 1,\end{equation} $h\in H$.
Again by  Lemma \ref{LemmaRefinedTracesAppr} we obtain
\begin{equation}\label{RefinedTracesAppr2} |tr \pi^A(\tilde h_i) - \lambda(\tilde h_i)| = |\chi^A(f^A(\tilde h_i)) - \lambda(\tilde h_i)| \le L \; max_{i\le N} |\chi^A(f^A(h_i)) - \lambda^A(h_i)|\le L\epsilon = \tilde\epsilon\end{equation} and also
\begin{equation}\label{RefinedTracesAppr3}|tr \pi^B(\phi(\tilde h_i)) - \lambda(\tilde h_i)|< \tilde \epsilon,\end{equation}
for any $i=1, \ldots, M.$

Since $f^A(a_i)\notin f^A(H)$ and $f^A(H)$ is central in $G^A$, it follows from definition of induced representations that all diagonal entries of $\pi^A(a_i)$ are zero, so that we have
\begin{equation}\label{RefinedTracesAppr4} tr \pi^A(a_i) =0.\end{equation}
Similarly \begin{equation}\label{RefinedTracesAppr5} tr \pi^B(b_i) =0.\end{equation}
By (\ref{RefinedTracesAppr1}), (\ref{RefinedTracesAppr01})
(\ref{RefinedTracesAppr2}), (\ref{RefinedTracesAppr3}), (\ref{RefinedTracesAppr4}), (\ref{RefinedTracesAppr5}) we are done.

\end{proof}

In particular case when $A=B$, $H=K$ we obtain

\begin{corollary}\label{CorollaryTracesInducedFromCenter} Let $A$  be a group and $H$ be its finitely generated central subgroup such that  there exists an $H$-filtration of $A$. Let $\lambda$ be a 1-dimensional representation of $H$,
$$\tilde\lambda (a)= \begin{cases} \lambda(a),\; a\in H \\ 0,\; a\notin  H \end{cases}.$$ Then there exist finite-dimensional representations $\pi_n, $ of $A$ such that for each $h\in H$, $\pi_n(h)$ is a  scalar operator  and
$$\tilde\lambda(a) = \lim tr\; \pi_n(a),$$ for any $a\in A$.
\end{corollary}

\begin{remark} By taking appropriate multiples of the representations $\pi_n^A, \pi_n^B$ in Theorem \ref{TracesInducedFromCenter} one can arrange them to live on the same space.
\end{remark}

\section{Auxiliary results on group C*-algebras}

Let $G$ be a discrete group and let $C$ be its normal subgroup such that $G/C$ is amenable.
For a given 1-dimensional representation $\lambda$ of $C$ we introduce a norm on $\mathbb C G$ by
$$\|f\|_{\lambda} = \sup \|\pi(f)\|$$
 where supremum is taken over all representations $\pi$ such that $\pi\;|_C = \lambda 1.$  We are going to prove that any u.c.p. map $\Phi$ from $C^*(G)$ to any unital $C^*$-algebra such that $\Phi\;|_C = \lambda 1$ factorizes through $\Lambda_{\tilde\lambda}$.
In particular case when $C$ is central, this will imply that
the $C^*$-algebra $\overline{\mathbb C G}^{\|\|_{\lambda}}$ obtained as the completion of $\mathbb C G$ by the norm $\|\;\|_{\lambda}$ coincides with $\Lambda_{\tilde\lambda}(C^*(G))$. In the case $C= \{e\}$ this gives us the well-known $C^*(G) = C^*_r(G)$.

\medskip

  To prove this we will modify the proof  of [\cite{Davidson}, Th. 7.2.8].

\begin{lemma}\label{MatrixElement}
Let $G$ be a discrete group, $C$ its subgroup, and $\lambda$ a one-dimensional representation of $C$. Let $\psi$ be a positive-definite function on $G$ supported in the union of finitely many cosets of $C$ and such that
$\psi(cg) = \lambda(c)\psi(g).$ Then there is a unit vector $\eta\in H_{\tilde\lambda}$ such that
$$\psi(g) = \langle\Lambda_{\tilde\lambda}(g)\eta, \eta\rangle,$$ for any $g\in G$.
\end{lemma}
\begin{proof} Fix some representatives $g_1, g_2, \ldots$ of the cosets of $C$. By Lemma \ref{basis} $g_i+N_{\tilde\lambda}$, $i\in \mathbb N$, form an orthonormal basis in $H_{\tilde\lambda}$.
We define a linear map $T_{\psi}$ on $H_{\tilde\lambda}$ by $$T_{\psi}(g_i+N_{\tilde\lambda}) = \sum_k\psi(g_k^{-1})(g_ig_k + N_{\tilde\lambda}),$$  $i\in \mathbb N$. Being a finite linear combination of the operators of right multiplication by $g_k$, which are unitary, $T_{\psi}$ is bounded. We now check that $T_{\psi}$ is positive. Indeed for any finite combination $\sum_i \xi_i (g_i+N_{\tilde\lambda})$ of basis vectors we have $$\langle T_{\psi}(\sum_i \xi_i(g_i+N_{\tilde\lambda})), \sum_k \xi_k(g_k+N_{\tilde\lambda})\rangle = \sum_{i, k, l} \xi_i\bar\xi_k \psi(g_l^{-1})\tilde\lambda(g_k^{-1}g_ig_l).$$
There is only one $l = l(i, k)$ such that $g_k^{-1}g_ig_l\in C$, and we can write $g_l = g_i^{-1}g_kc_{i, k}.$ Thus
\begin{multline*}\langle T_{\psi}(\sum_i \xi_i(g_i+N_{\tilde\lambda})), \sum_k \xi_k(g_k+N_{\tilde\lambda})\rangle = \sum_{i, k} \xi_i\bar\xi_k \psi(c_{i, k}^{-1}g_k^{-1}g_i)\lambda(c_{i, k})\\  = \sum_{i, k} \xi_i\bar\xi_k \psi(g_k^{-1}g_i)\lambda(c_{i, k}^{-1})\lambda(c_{i, k}) = \sum_{i, k} \xi_i\bar\xi_k \psi(g_k^{-1}g_i) \ge 0.\end{multline*} Hence $T_{\psi}$ is positive. Next we notice that being a linear combination of right multiplication operators $T_{\psi}$ commutes with $\Lambda_{\tilde\lambda}(G)$.
Let $\eta = T_{\psi}^{1/2}(e+N_{\tilde\lambda})$. Then $$ \langle \Lambda_{\tilde\lambda}(s)\eta, \eta\rangle = \langle s+N_{\tilde\lambda}, \sum_k \psi(g_k^{-1})(g_k+N_{\tilde\lambda})\rangle =
\sum_k\overline{\psi(g_k^{-1})}\tilde\lambda(g_k^{-1}s),$$ $s\in G$.
There is only one $k$ such that $g_k^{-1}s\in C$. For this $k$ we can write $g_k^{-1}s =c$ whence $g_k = sc^{-1}.$
Also as $\psi$ is positive-definite we have $\psi(g_k^{-1})= \overline{\psi(g_k)}$ and $\psi(s^{-1}c) = \lambda(c^{-1})\psi(g)$.  Hence
$$\langle \Lambda_{\tilde\lambda}(s)\eta, \eta\rangle = \psi(sc^{-1})\lambda(c) = \psi(s)\lambda(c^{-1})\lambda(c) = \psi(s),$$ $s\in G$.
Finally, as $\tilde\lambda(g_k) \neq 0$ only when $g_k=e$, we obtain \begin{multline*}\|\eta\|^2= \langle T_{\psi}^{1/2}(e+N_{\tilde\lambda}), T_{\psi}^{1/2}(e+N_{\tilde\lambda})\rangle = \langle T_{\psi}(e+N_{\tilde\lambda}), e+N_{\tilde\lambda}\rangle \\ = \sum_k \psi(g_k^{-1})\langle g_k+ N_{\tilde\lambda}, e+N_{\tilde\lambda}\rangle = \sum_k \psi(g_k^{-1})\tilde\lambda(g_k) = 1.\end{multline*}

\end{proof}

 Recall that for a u.c.p. map $\Phi:A \to B$ its  {\it multiplicative domain} is the set
 $$\it M_{\Phi} = \{a\in A\;|\; \Phi(ab) = \Phi(a)\Phi(b), \Phi(ba) = \Phi(b)\Phi(a),  \forall b\in A\}.$$

\begin{theorem}\label{factorization} Let $G$ be a discrete group, let $C$ be its normal subgroup such that $G/C$ is amenable, and let $\lambda$ be a 1-dimensional representation of $C$. Let $A$ be a unital $C^*$-algebra and $\Phi: C^*(G)\to A$ be a u.c.p. map such that $\Phi(c)  = \lambda(c) 1$, for any $c\in C$. Then
$$\|\Phi(f)\| \le \| \Lambda_{\tilde\lambda}(f)\|,$$ for any $f\in C^*(G)$.
\end{theorem}
\begin{proof}
At first we will prove the statement for any positive $f\in \mathbb C G$. Embed $A$ into $\mathcal B(H)$. Since $\Phi(f)$ is a positive operator,
 $\|\Phi(f)\| = \sup_{\|\xi\|\le 1} (\Phi(f)\xi, \xi)$. Hence it would be sufficient to prove that  $|(\Phi(f)\xi, \xi)|\le \|\Lambda_{\tilde\lambda}(f)\|$, for any $\xi \in H$. So let us fix $\xi$. Define a positive-definite function $\phi$ on $G$ by
 $$\phi(g) = (\Phi(g)\xi, \xi).$$ Since $G/C$ is amenable, there are finitely supported positive-definite functions $\phi_n$ on $G/C$ which pointwisely converge to $1$. Let $\tilde\phi_n(g) = \phi_n(gC).$ Then
 \begin{equation}\label{converge} (\phi\tilde\phi_n)(g) = \phi(g)\phi_n(gC) \to_{n\to\infty} \phi(g),\end{equation} for each $g\in G$.  Since the set of all positive-definite functions which take value 1 at the unit  is closed under pointwise multiplication  (\cite{Davidson}, Lemma VII.2.6) and since for each $n$ the function $\tilde\phi_n$ is supported on the union of finitely many cosets, we conclude that for each $n$ the function $\phi\tilde\phi_n$ is positive-definite and supported on the union of finitely many cosets. By Choi's theorem \cite{Choi} the multiplicative domain $\it M_{\Phi}$ coincides with the set $$S = \{a\in C^*(G)\;|\; \Phi(a^*a) = \Phi(a)^*\Phi(a), \Phi(aa^*) = \Phi(a)\Phi(a)^*\}.$$ It follows from the assumptions that $C^*(C) \subseteq S$ and hence $\Phi(cg) = \lambda(c)\Phi(g)$, for any $c\in C$, $g\in G$. Therefore
 $$(\phi\tilde\phi_n)(cg) = (\Phi(cg)\xi, \xi) \phi_n(cgC) =
 \lambda(c)(\Phi(g)\xi, \xi) \phi(gC) =\lambda(c)(\phi\tilde\phi_n)(g),$$ for any $c\in C$, $g\in G$. By Lemma \ref{MatrixElement} for each $n$ there is a unit vector $\eta_n\in H_{\tilde\lambda}$ such that $$(\phi\tilde\phi_n)(g) = \langle \Lambda_{\tilde\lambda}(g) \eta_n, \eta_n \rangle, $$ for any $g\in G$. Therefore
 \begin{multline*}|(\Phi(f)\xi, \xi)| = | \sum_{g\in supp(f)} f(g)(\Phi(g)\xi, \xi)| =
| \sum_{g\in supp(f)} f(g)\phi(g)|\\ = |\lim_{n\to \infty} \sum_{g\in supp(f)} f(g)(\phi\tilde\phi_n)(g)| = |\lim_{n\to \infty}\sum_{g\in supp(f)} f(g) \langle \Lambda_{\tilde\lambda}(g) \eta_n, \eta_n \rangle| \\= |\lim_{n\to \infty}  \langle \Lambda_{\tilde\lambda}(f) \eta_n, \eta_n \rangle| \le \|\Lambda_{\tilde\lambda}(f)\|.\end{multline*}
Thus the statement is proved for any positive $f\in \mathbb C G$.  Using Schwarz inequality for u.c.p. maps we obtain for  arbitrary $f\in \mathbb C G$
$$\|\Phi(f)\|^2  = \|\Phi(f)^*\Phi(f)\| \le \|\Phi(f^*f)\| \le \Lambda_{\tilde\lambda}(f^*f)\| = \|\Lambda_{\tilde\lambda}(f)\|^2.$$
As $\Phi$ is continuous, by approximating  elements of $C^*(G)$ by elements of $\mathbb C G$ it is straightforward to obtain that $\|\Phi(f)\| \le \| \Lambda_{\tilde\lambda}(f)\|,$ for any $f\in C^*(G)$.
\end{proof}

\begin{lemma}\label{scalar} Let $B\subseteq A$ be a $C^*$-subalgebra and let $\tau$ be a state on $A$ such that $\tau\;|_B$ is $\ast$-multiplicative. Assume that either 1) $B$ is central or 2) $\tau$ is a trace.  Then $\Lambda_{\tau}(b) = \tau(b) 1$, for each $b\in B$.
\end{lemma}
\begin{proof} As $\Lambda_{\tau}(b) (a+ N_{\tau}) = ab + N_{\tau}$, we need to check that $ba- \tau(b)a \in N_{\tau}$, for each $a\in A$. We have
\begin{multline*} \tau \left((ba- \tau(b)a)^*(ba- \tau(b)a)\right) = \tau \left((ba- \tau(b)a)(ba- \tau(b)a)^*\right) \\ = \tau\left((b-\tau(b)1)aa^*(b-\tau(b)1)^*\right) \le \|a\|^2
\tau\left((b-\tau(b)1)(b-\tau(b)1)^*\right)=0
\end{multline*} which means that $ba- \tau(b)a \in N_{\tau}$, for each $a\in A$.
\end{proof}

\begin{corollary}\label{kernel} Let $G$ be a discrete group, let $C$ be its normal subgroup such that $G/C$ is amenable, and let $\lambda$ be a 1-dimensional representation of $C$. Then for any unital $C^*$-algebra $A$ and any u.c.p. $\Phi: C^*(G)\to A$ such that $\Phi(c)  = \lambda(c) 1$, for any $c\in C$,   the following holds:

(i) $Ker \Lambda_{\tilde\lambda} \subseteq  Ker \Phi$.

\noindent If $C$ is central, then the following also holds:

(ii) $\overline{\mathbb C G}^{\|\|_{\lambda}} = \Lambda_{\tilde\lambda}(C^*(G))$.

\end{corollary}
\begin{proof} (i) follows from Theorem \ref{factorization}.

\noindent (ii): The inclusion $\overline{\mathbb C G}^{\|\|_{\lambda}} \subseteq \Lambda_{\tilde\lambda}(C^*(G))$ follows from Theorem \ref{factorization} (even if $C$ is not central).
In the case when $C$ is central, by Lemma \ref{scalar} applied to $A = C^*(G)$ and $B = C^*(C)$, we obtain that  $\Lambda_{\tilde\lambda}|_C = \lambda 1$ and hence the inclusion becomes an equality.
\end{proof}

Although the results above fail without the amenability assumption, one can prove an algebraic analogue of the statement (i) of Corollary \ref{kernel}. We write it below for completeness.

\begin{proposition}\label{algebra} Let $G$ be a discrete group, $C$  its subgroup, $ \lambda$ a 1-dimensional representation of $C$ and $\rho$ a $\ast$-representation of $C^*(G)$ such that $\rho\;|_C = \lambda 1$.  Then $Ker \Lambda_{\tilde\lambda} \bigcap \mathbb C G \subseteq Ker \rho \bigcap \mathbb C G$.
\end{proposition}
\begin{proof} Suppose that $\sum t_g g\in Ker \Lambda_{\tilde\lambda} \bigcap \mathbb C G.$ It follows from the GNS-construction that
$\tilde \lambda((\sum t_gg)^*(\sum t_gg)) = 0$. Therefore we obtain
$$0 = \tilde \lambda((\sum t_gg)^*(\sum t_gg)) = \sum \bar{t_g}t_h \tilde\lambda(g^{-1}h) = \sum_{g^{-1}h\in C}\bar{t_g}t_h\lambda(g^{-1}h).$$
Hence \begin{equation}\label{1}\sum_{g^{-1}h\in C}\bar{t_g}t_h\rho(g^{-1}h) = \sum_{g^{-1}h\in C}\bar{t_g}t_h\lambda(g^{-1}h) 1 =0.\end{equation}
Let $g_1C, g_2C, \ldots$ be the cosets for $C$. Let $$L_i = \sum_{g\in g_iC} t_g\rho(g).$$
Since
$g^{-1}h\in C$ if and only if $g, h$ belong to the same coset, using (\ref{1}) we obtain
$$\sum_i L_i^*L_i = \sum_i \sum_{g, h\in g_iC} \bar{t}_gt_h\rho(g^{-1}h) = \sum_{g^{-1}h\in C} \bar{t}_gt_h\rho(g^{-1}h) = 0.$$
Hence $L_i=0$ for each i. Since $G = g_1C\sqcup g_2C \sqcup \ldots$ we have $$\rho(\sum t_gg) = \rho(\sum_i\sum_{g\in g_iC} t_gg) = \sum_i L_i =0.$$
\end{proof}

\section{When central amalgamated free products are RFD}

The following statement is extracted from [\cite{ExelLoring}, proof of $(a)\Rightarrow (b)$ in Th. 2.4].

\begin{lemma}\label{ExelLoring} (Exel-Loring \cite{ExelLoring})
Let $f_n$, $n\in \mathbb N$,  and  $f$ be states on a C*-algebra $A$ such that  $f_n\to f$ $\ast$-weakly. Then there exist coisometries $V_n: H_f\to H_{f_n}$ such that for any $a\in A$, $\Lambda_f(a)$ is a pointwise SOT-limit of the (degenerate) representations $V_n^*\Lambda_{f_n}(a) V_n$:
$$\Lambda_f(a) = SOT-\lim V_n^*\Lambda_{f_n}(a) V_n.$$
\end{lemma}

The following proposition and corollary are due to Don Hadwin. They seem to never have been published (a somewhat close statement was published in [\cite{Don}, Th. 4.3]. For this reason a brief proof of the proposition was included in K. Courtney and the author's paper \cite{KristinTanya}.

\begin{proposition}[Hadwin] Let $\{e_n\}$ be an orthonormal basis in an infinite-dimensional separable Hilbert space $H$, let $A, B, C, D \in \mathcal B(H)$ and  $T = \left(\begin{array}{cc} A&B\\C&D\end{array}\right)$. Then for any unitary $w_n: H \to H\oplus H$ such that $w_ne_k = (e_k, 0)$, $1\le k \le n$,
$w_n^*Tw_n$ converge to $A$ in the weak operator topology.
Moreover we have convergence in the strong operator topology if and only if $C=0$ and in the $\ast$-strong operator topology if and only if $C=B=0$. 
\end{proposition}

Since the unitaries $w_n$ in the proposition above do not depend on the operators $A, B, C, D$, and since all infinite-dimensional separable Hilbert spaces are isomorphic, we obtain the following corollary.

\begin{corollary} [Hadwin]\label{Don} Let  $\pi$ and $\rho$ be representations of a C*-algebra on infinite-dimensional separable spaces $H$ and $K$ respectively. Then there are unitaries $w_n: H \to H\oplus K$ such that $\pi = SOT-\lim w_n^*(\pi\oplus \rho) w_n.$
\end{corollary}

The other lemmas we need are all very easy.


\begin{lemma}\label{GNSfin-dim} Let $\rho$ be a finite-dimensional representation of $A$ and let $f$ be a state on $\rho(A)$. Then $\Lambda_{f\circ \rho}$ is finite-dimensional.
\end{lemma}
\begin{proof} We have $N_{f\circ \rho} = \{a\in A\;|\; f(\rho(a^*a))=0\} \supseteq \ker \rho.$ Therefore
$$\dim (A/N_{f\circ \rho}) \le \dim (A/{\ker \rho}) < \infty$$ and hence $H_{f\circ \rho}$ is finite-dimensional.
\end{proof}

\begin{lemma}\label{infty} Let $\rho$ and $\rho_n$, $n\in \mathbb N$, be (possibly degenerate) representations of a $C^*$-algebra $A$ on $H$ such that $$\rho(a) = SOT-\lim \rho_n(a),$$ for any $a\in A$. Then there are $m_n\in \mathbb N$ such that $$\rho^{(\infty)}(a) = SOT-\lim \rho_n^{\oplus m_n}(a),$$ for any $a\in A$.
\end{lemma}
\begin{proof} Given $a_1, \ldots, a_N\in A$, $\xi_1, \ldots, \xi_N\in H^{(\infty)}$ and $\epsilon >0$, we need to find $m, n\in \mathbb N$ such that
$$\|\rho^{(\infty)}(a_l)\xi_i - \rho_n^{\oplus m}(a_l)\xi_i\|\le 3\epsilon,$$ for any $i, l \le N.$

Let $M = \max\{\|a_i\|, i\le N\}.$
For each vector $\xi_i = (\xi_i^{(1)}, \xi_i^{(2)}, \ldots) \in H^{(\infty)}$, let $\xi_{i,m}$ denote the vector $(\xi_i^{(1)}, \xi_i^{(2)}, \ldots, \xi_i^{(m)}, 0, 0, \ldots) \in H^{(\infty)}$.  There exists $m\in \mathbb N$ such that $$\|\xi_i - \xi_{i, m}\|\le \epsilon/M, $$ for any $i\le N$.
By our assumptions there exists $n\in \mathbb N$ such that $$\|\rho(a_l)\xi_i^{(j)} - \rho_n(a_l)\xi_i^{(j)}\|\le \epsilon/\sqrt N,$$ for any $j\le m, i\le N, l\le N.$
Then for any $l, i\le N$ we have
\begin{multline*} \|\rho^{(\infty)}(a_l)\xi_i - \rho_n^{\oplus m}(a_l)\xi_i\| \le \|\rho^{(\infty)}(a_l)(\xi_i -\xi_{i, m})\| + \\ \|\rho^{(\infty)}(a_l)\xi_{i,m} - \rho_n^{\oplus m}(a_l)\xi_{i,m}\| +  \|\rho_n^{\oplus m}(a_l)(\xi_i -\xi_{i, m})\| \le 3\epsilon.
\end{multline*}
\end{proof}

Let  $A$ and $B$ be C*-algebras, $C_1,$ $C_2$ be their C*-subalgebras respectively, and $\phi: C_1\to C_2$ be an isomorphism. Let $\ast$-homomorphisms $\psi_A: A \to D$ and $\psi_B: B \to D$ be such that $\psi_A(c) = \psi_B(\phi(c))$, for any  $c\in C_1$.
The corresponding $\ast$-homomorphism from $A\ast_{C_1\cong^{\phi} C_2} B$ to $D$ will be denoted by $\sigma_{\psi_A, \psi_B}.$

\begin{lemma}\label{SOT-limit} Let $A$ and $B$ be C*-algebras and $C_1$, $C_2$ be their $C^*$-subalgebras respectively and let $\phi: C_1 \to C_2$ be an isomorphism. Let $\pi_1, \pi_{1, n}$ be representations of $A$, $\pi_2, \pi_{2, n}$ be representations of $B$,  such that $$\pi_1(c)=\pi_2(\phi(c)), \;\; \pi_{1, n}(c)=\pi_{2, n}(\phi(c)),$$  for any $c\in C_1$, $n\in \mathbb N$, and  $\pi_1 = \ast-SOT-\lim \pi_{1, n}$, $\pi_2 = \ast-SOT-\lim \pi_{2, n}$. Suppose that $x\in A\ast_{C_1\cong^{\phi}C_2} B$ and $\sigma_{\pi_1, \pi_2}(x)\neq 0$. Then there is $n\in \mathbb N$ such that $\sigma_{\pi_{1,n}, \pi_{2, n}}(x)\neq 0$.
\end{lemma}
\begin{proof} The C*-algebra $A\ast_{C_1\cong^{\phi}C_2}B$ is the closure of the span of  monomials of the form $y\! = \!i_A(a_1)i_B(b_1)i_A(a_2)\ldots$ and $y = i_B(b_1)i_A(a_2)....$ to which we further refer simply as monomials.

{\bf Claim.} For any $\epsilon >0$, a monomial $y\in A\ast_{C_1\cong^{\phi}C_2}B$, and $\xi\in H$, there is $n\in \mathbb N$ such that
$\|\sigma_{\pi_{1, n}, \pi_{2, n}}(y) \xi - \sigma_{\pi_1, \pi_2}(y)\xi\| \le \epsilon.$

{\it Proof of Claim.} We will prove the claim by induction on the length of a monomial. Suppose it is proved for monomials of length $k$ and let $y$ be a monomial of length $k+1$. Then $y$ can be written either as $y = y_1 i_A(a)$ or as $y = y_1 i_B(b)$, where $y_1$ is a  monomial of of length $k$. Say $y = y_1 i_A(a)$. Since $\pi_1 = \ast-SOT-\lim \pi_{1, n}$, there is $N$ such that for any $n\ge N$ we have
$\|(\pi_{1, n}(a) - \pi_1(a))\xi\|\le \epsilon/2.$ By induction assumption applied to $\epsilon$, $y_1$, and $\pi_1(a)\xi$, there is $N'>N$ such that for any $n\ge N'$ we have $\|\sigma_{\pi_{1}, \pi_{2}}(y_1)\pi_1(a)\xi - \sigma_{\pi_{1, n}, \pi_{2, n}}(y_1)\pi_1(a)\xi\|\le \epsilon/2.$ Hence
\begin{multline*}\|\sigma_{\pi_{1, n}, \pi_{2, n}}(y) \xi - \sigma_{\pi_1, \pi_2}(y)\xi\| = \|\sigma_{\pi_{1, n}, \pi_{2, n}}(y_1)\pi_{1, n}(a)\xi - \sigma_{\pi_1, \pi_2}(y_1)\pi_1(a)\xi\| \\ \le \|\sigma_{\pi_{1, n}, \pi_{2, n}}(y_1)\pi_{1}(a)\xi - \sigma_{\pi_1, \pi_2}(y_1)\pi_1(a)\xi\| + \|\sigma_{\pi_{1, n}, \pi_{2, n}}(y_1)(\pi_{1, n}(a) - \pi_1(a))\xi\| \le  \epsilon.\end{multline*}
Claim is proved.

Now one easily deduces from the claim that $\sigma_{\pi_{1}, \pi_{2}} = SOT-\lim \sigma_{\pi_{1, n}, \pi_{2, n}}$ which in its turn easily implies the statement of the lemma. Indeed, we find $\xi\in H$ such that $\|\sigma_{\pi_1, \pi_2}(x)\xi\|\ge \|\sigma_{\pi_1, \pi_2}(x)\|/2$ and  a monomial $y$ such that  $\|x-y\|\le \|\sigma_{\pi_1, \pi_2}(x)\|/8$. By the claim there is $n$ such that $$\|\sigma_{\pi_{1, n}, \pi_{2, n}}(y) \xi - \sigma_{\pi_1, \pi_2}(y)\xi\| \le
\|\sigma_{\pi_1, \pi_2}(x)\|/8.$$ Hence
\begin{multline*} \|\sigma_{\pi_{1, n}, \pi_{2, n}}(x)\|\ge \|\sigma_{\pi_{1, n}, \pi_{2, n}}(y)\xi\| - \|\sigma_{\pi_{1, n}, \pi_{2, n}}(y-x)\xi\|\\ \ge \|\sigma_{\pi_1, \pi_2}(y)\xi\| - \|\sigma_{\pi_{1, n}, \pi_{2, n}}(y) \xi - \sigma_{\pi_1, \pi_2}(y)\xi\|  - \|\sigma_{\pi_1, \pi_2}(x)\|/8
\\ \ge \|\sigma_{\pi_1, \pi_2}(x)\xi\| - \|\sigma_{\pi_1, \pi_2}(x-y)\xi\| - \|\sigma_{\pi_1, \pi_2}(x)\|/4 \ge \|\sigma_{\pi_1, \pi_2}(x)\|/8 \neq 0.
\end{multline*}

\end{proof}

\begin{theorem}\label{main} Let $G_1, G_2$ be amenable discrete groups, $C_1$ and $C_2$ be their finitely generated central subgroups respectively, and $\phi: C_1 \to C_2$ an isomorphism such that there exist $(C_1, C_2, \phi)$-compatible filtrations of $G_1$ and $G_2$. Then $C^*(G_1\ast_{C_1\cong^{\phi} C_2} G_2)$ is RFD.
\end{theorem}


\begin{proof} Let $0\neq x\in C^*(G_1\ast_{C_1\cong^{\phi} C_2} G_2).$ We are going to find a finite-dimensional representation of $C^*(G_1\ast_{C_1\cong^{\phi} C_2} G_2)$ which does not vanish on $x$.  Let $\pi$ be an irreducible representation of $C^*(G_1\ast_{C_1\cong^{\phi} C_2} G_2)$  on a Hilbert space $K$ such that $\pi(x)\neq 0$. We can assume that $\dim K = \infty$. Since $\pi$ is irreducible, there is a 1-dimensional representation $\lambda$ of $C_1$ such that $\pi(c) = \lambda(c) 1 = \pi(\phi(c)), $ for each $c\in C_1$. Let
 $$\tilde\lambda^{(1)} (g)= \begin{cases} \lambda(g),\; g\in C_1 \\ 0,\; g\in G_1\setminus C_1 \end{cases}, \;\tilde\lambda^{(2)} (g)= \begin{cases} \lambda(\phi^{-1}(g)),\; g\in C_2 \\ 0,\; g\in G_2\setminus C_2 \end{cases}.$$
 Let $\Lambda_{\tilde\lambda^{(i)}}$, $i=1, 2$, be the GNS-representation of $G_i$ with respect to the trace $\tilde\lambda^{(i)}$. The corresponding Hilbert spaces will be denoted by $H_i$, $i=1, 2$.

\medskip

{\bf Claim 1:} There exist unitaries $U_i\in B(K, H_i^{(\infty)})$, $i=1, 2$, such that
$$U_1^*\left(\Lambda_{\tilde\lambda^{(1)}}\right)^{(\infty)}(c)U_1 = U_2^*\left(\Lambda_{\tilde\lambda^{(2)}}\right)^{(\infty)}(\phi(c))U_2,$$ for any $c\in C_1$, and  $$\sigma_{U_1^*\left(\Lambda_{\tilde\lambda^{(1)}}\right)^{(\infty)}U_1, \; U_2^*\left(\Lambda_{\tilde\lambda^{(2)}}\right)^{(\infty)}U_2} (x) \neq 0.$$

Proof of Claim 1: By Corollary \ref{kernel}, $\ker \pi|_{C^*(G_i)} \supseteq \ker \Lambda_{\tilde\lambda^{(i)}}$, $i=1, 2$. Hence $$ rank \left(\pi|_{C^*(G_i)}\oplus \left(\Lambda_{\tilde\lambda^{(i)}}\right)^{\infty}\right)(a) = rank \left(\Lambda_{\tilde\lambda^{(i)}}\right)^{\infty}(a),$$ for any $a\in C^*(G_i)$. By Vociulescu's theorem the representations $\pi|_{C^*(G_i)}\oplus \left(\Lambda_{\tilde\lambda^{(i)}}\right)^{\infty}$ and $\left(\Lambda_{\tilde\lambda^{(i)}}\right)^{\infty}$ are approximately unitarily equivalent. On the other hand, by Corollary \ref{Don} the representation $\pi|_{C^*(G_i)}$ is the $\ast$-strong limit of unitary conjugates of $\pi|_{C^*(G_i)}\oplus \left(\Lambda_{\tilde\lambda^{(i)}}\right)^{\infty}$. All together this says us that $\pi|_{C^*(G_i)}$ is the $\ast$-strong limit of unitary conjugates of $\left(\Lambda_{\tilde\lambda^{(i)}}\right)^{\infty}$. We notice also that since by Lemma \ref{scalar}
$$\left(\Lambda_{\tilde\lambda^{(1)}}\right)^{\infty}(c) = \tilde\lambda^{(1)}(c)1 = \lambda(c)1 =  \tilde\lambda^{(2)}(\phi(c))1 = \left(\Lambda_{\tilde\lambda^{(2)}}\right)^{\infty}(\phi(c)),$$ the same is true for any unitary conjugates of $\left(\Lambda_{\tilde\lambda^{(i)}}\right)^{\infty}$. Claim follows now from Lemma \ref{SOT-limit}.

\medskip

\medskip

{\bf Claim 2:}  There exist finite-dimensional representations $\pi_n^{(i)} $ of $G_i$, $i=1, 2$, such that for any $c\in C_1$,
$\pi_n^{(1)}(c)$ and $\pi_n^{(2)}(\phi(c))$ are scalar matrices (of possibly different sizes) with the corresponding scalars being equal to each other, and \begin{equation} \tilde\lambda^{(i)}(a) = \lim_{n\to \infty} tr \pi_n^{(i)}(a),\end{equation} for any $a\in C^*(G_i)$.

Proof of Claim 2: By Theorem \ref{TracesInducedFromCenter} there exist finite-dimensional representations $\pi_n^{(i)} $ of $G_i$, $i=1, 2$, such that for any $c\in C_1$,
$\pi_n^{(1)}(c)$ and $\pi_n^{(2)}(\phi(c))$ are scalar matrices with the corresponding scalars being equal to each other, and \begin{equation} \tilde\lambda^{(i)}(g) = \lim_{n\to \infty} tr \pi_n^{(i)}(g),\end{equation} for any $g\in G_i$. As $\mathbb C G_i$ is dense in $C^*(G_i)$ we conclude that \begin{equation} \tilde\lambda^{(i)}(a) = \lim_{n\to \infty} tr \pi_n^{(i)}(a),\end{equation} for any $a\in C^*(G_i)$. Claim 2 is proved.

\medskip

By Lemma \ref{ExelLoring} there exist coisometries $V_n^{(i)}: H_i \to H_{tr {\pi_n^{(i)}}}$ such that for the representations $\rho_n^{(i)} = {V_n^{(i)}}^*\Lambda_{tr {\pi_n^{(i)}}} V_n^{(i)}$ we have
$$\Lambda_{\tilde\lambda^{(i)}}(a) = SOT-\lim \rho_n^{(i)}(a),$$ $a\in C^*(G_i)$.
By Lemma \ref{GNSfin-dim}, for each $n\in \mathbb N$, $\Lambda_{tr {\pi_n^{(i)}}}$,  $i=1, 2$,  are finite-dimensional representations (size might depend on $i$) and by Lemma \ref{scalar}, $\Lambda_{tr{\pi_n^{(1)}}}(c)$ and $\Lambda_{tr{\pi_n^{(2)}}}(\phi(c))$,    are scalar matrices of possibly different sizes but with the same scalars. It follows that the same is true for  $\rho_n^{(i)}$, $i=1, 2$. By Lemma \ref{infty} there are finite multiples of $\rho_n^{(i)}$, let us call them $\tilde \rho_n^{(i)}$, such that
\begin{equation}\label{uparrow100}{\Lambda_{\tilde\lambda^{(i)}}}^{(\infty)}(a) = SOT-\lim \tilde\rho_n^{(i)}(a),\end{equation} and we still have  that for each $c\in C_1$,  $\tilde\rho_n^{(1)}(c)$ and $\tilde\rho_n^{(2)}(\phi(c))$  are scalar matrices of possibly different size but with the same scalars. Let us denote by $H_n^{(i)}\subset H_i^{(\infty)}$ the essential spaces of the representations $\tilde \rho_n^{(i)}$. It follows from (\ref{uparrow100}) that $SOT\!-\!\lim P_{H_n^{(i)}} =~1_{H_i^{(\infty)}}$. Now we define a finite-dimensional subspace $K_n^{(i)}$ of $K$
by $$K_n^{(i)} = U_i^{-1}(H_n^{(i)}),$$ (here $U_i$'s are the unitaries from claim 1). We define a unitary $U_{n, i}: K_n^{(i)}\to H_n^{(i)}$  by
$$U_{n, i} = U_i|_{K_n^{(i)}}.$$ Then  $U_i = \ast-SOT-\lim U_{n, i}$ and we obtain  \begin{equation}\label{MainNewEq}U_i^*{\Lambda_{\tilde\lambda^{(i)}}}^{(\infty)}(a)U_i = SOT-\lim U_{n, i}^*\tilde\rho_n^{(i)}(a)U_{n, i},\end{equation}
 and we still have  that for each $c\in C_1$,  the finite-dimensional operators  $U_{n, 1}^*\tilde\rho_n^{(1)}(c)U_{n, 1}$ and $U_{n, 2}^*\tilde\rho_n^{(2)}(\phi(c))U_{n, 2}$  are scalar and the corresponding scalars coincide.

 Now for each $n$ we choose a finite-dimensional subspace $\tilde K_n$ of $K$ containing both the essential subspace of $U_{n, 1}^*\tilde\rho_n^{(1)}U_{n, 1}$ and of $U_{n, 2}^*\tilde\rho_n^{(2)}U_{n, 2}$  with dimension of $\tilde K_n$ being a common multiple of $\dim U_{n, 1}^*\tilde\rho_n^{(1)}U_{n, 1}$ and $\dim U_{n, 2}^*\tilde\rho_n^{(2)}U_{n, 2}$. We take   appropriate multiples of $U_{n, 1}^*\tilde\rho_n^{(1)}U_{n, 1}$ and $U_{n, 2}^*\tilde\rho_n^{(i)}U_{n, 2}$ on $\tilde K_n$ and we call them $\bar\rho_n^{(i)},$ $i=1, 2$. Then $\bar\rho_n^{(i)}$, $i=1, 2$, live on the same finite-dimensional space $\tilde K_n$, $\bar\rho_n^{(1)}(c) = \bar \rho_n^{(2)}(\phi(c))$ (and is  a scalar operator on $\tilde K_n$ but it will not be used anymore), for any $c\in C_1$, and since they are extensions of $U_{n, i}^*\tilde\rho_n^{(i)}U_{n, i}$, by (\ref{MainNewEq}) we still have (see e.g. [\cite{ExelLoring}, Lemma 3.1]) $$U_i^*{\Lambda_{\tilde\lambda^{(i)}}}^{(\infty)}(a)U_i = SOT-\lim \bar\rho_n^{(i)}(a),$$ for any $a\in C^*(G_i)$. By the claim 1 and by Lemma \ref{SOT-limit} there is $n\in \mathbb N$ such that $\sigma_{\rho_n^{(1)}, \rho_n^{(2)}}(x)\neq 0$.
\end{proof}

\begin{corollary}\label{CorollaryMainCases} Let $G_1, G_2$ be amenable discrete groups, $C_1$ and $C_2$ be their finitely generated central subgroups respectively, and $\phi: C_1 \to C_2$ an isomorphism. Then $C^*(G_1\ast_{C_1\cong^{\phi} C_2} G_2)$ is RFD if any of the following conditions holds:

(i) $G_1$, $G_2$ are polycyclic-by-finite;

(ii) $G_1$ is RF, $G_1/C_1$ is RF, $G_2$ is polycyclic-by-finite;


(iii) $G_1$ is RF, $C_1 = Z(G_1)$, $G_2$ is polycyclic-by-finite;


(iv) $G_1=G_2$ is RF, $C_1 = C_2$, $\phi=id$, $G_1/C_1$ is RF;

(v) $G_1=G_2$ is RF, $C_1 = Z(G_1) = C_2$, $\phi=id$.

\end{corollary}
\begin{proof}
(i), (ii),  (iii) follow  from Theorem \ref{main} and Corollary \ref{WhenCompFiltrExists2}.




\noindent (iv): By Corollary \ref{QuotientRF} there is a $C_1$-filtration of $G_1$. Taking the same filtration of $G_2$ we obtain compatible filtrations and  the statement follows from Theorem \ref{main}.

\noindent (v) follows from (iv) and Corollary \ref{QuotientCenter}.

\end{proof}

With an almost identical to Theorem \ref{main}  proof one can prove the following, formally even slightly stronger, statement.

\begin{theorem}\label{NewTheorem} Let $G_1, G_2$ be amenable discrete groups, and $C$ be a finitely-generated central subgroup in both of them, such that $G_1/C$ and $G_2/C$ are RF.  If $G_1\ast_{C} G_2$ is RF, then $C^*(G_1\ast_{C} G_2)$ is RFD.
\end{theorem}




In \cite{BekkaLouvet} Bekka and Louvet proved that an amenable group is RFD if and only if it is MAP. This, together with Proposition \ref{MAPhenceRF}, implies that a finitely generated amenable group is RFD if and only if it is RF. The corollaries below produce new classes of groups for which the properties RF and RFD are equivalent. Bekka and Louvet's result in the case of finitely generated groups can be obtained from the first corollary below by taking one of the groups and the amalgamating subgroup trivial.



\begin{corollary}\label{NewTheorem2} Let $G_1, G_2$ be finitely-generated amenable groups,  $C$ -- a finitely-generated central subgroup in both, such that $G_1/C$ and $G_2/C$ are RF.  Then \newline $C^*(G_1~\ast_{C}~G_2)$ is RFD if and only if $G_1\ast_{C} G_2$ is RF.
\end{corollary}
\begin{proof} The "if" part is proved in Theorem \ref{NewTheorem}.

"Only if": If $C^*(G_1\ast_{C} G_2)$ is RFD, then $G_1\ast_{C} G_2$ is MAP. Since $G_1\ast_{C} G_2$ is finitely generated, it is  RF by Proposition \ref{MAPhenceRF}.
\end{proof}

\begin{corollary} Let $G_1, G_2$ be finitely generated amenable groups with isomorphic finitely generated centers. Then $C^*(G_1\ast_{Z(G_1)\cong Z(G_2)} G_2)$ is RFD if and only if $G_1\ast_{Z(G_1)\cong Z(G_2)} G_2$ is RF.
\end{corollary}
\begin{proof} This follows from Corollary \ref{NewTheorem2} and Corollary \ref{QuotientCenter}.
\end{proof}


\section{HNN-extensions}

\subsection{HNN-extensions with central identical associated subgroups}
Here we will give a full characterization of when HNN-extensions of the form $\langle G, t\;|\; t^{-1}Ct = C, \; id\rangle$, with $G$ being finitely generated amenable and $C$ being central and finitely generated, are RFD.

Let $\pi: A\to \mathcal B(H)$ be a representation of $A$, $U\in \mathcal B(H)$ a unitary such that $U^{-1}\pi(b)U = \pi(\phi(b))$, for any $b\in B$. The corresponding representation of the HNN-extension $\langle A, t\;|\; t^{-1}Bt = D, \; \phi\rangle$  will be  denoted by $\sigma_{\pi, U}$.

\begin{lemma}\label{HNNlimit} Let $A$ be a $C^*$-algebra, $B$ and $D$ its C*-subalgebras, and $\phi: B\to D$ an isomorphism. Let $\pi: A\to \mathcal B(H)$ be a representation of $A$ and let $P_n\in \mathcal B(H)$ be projections such that $SOT\!-\!\lim P_n = 1$. Let  $\pi_n: A \to B(P_nH)$ be  representations of $A$ such that $\pi= SOT-\lim \pi_n$ and let  $V\in \mathcal B(H)$ and $V_n\in B(P_nH)$ be unitaries such that $V= \ast\!-\!SOT\!-\!\lim V_n$,  $V^{-1}\pi(b)V = \pi(\phi(b))$, $V_n^{-1}\pi_n(b)V_n = \pi_n(\phi(b))$, for any $b\in B$, $n\in \mathbb N$. Let $x$  be an element of the HNN-extension $\langle A, t\;|\; t^{-1}Bt = D, \; \phi\rangle$ such that $\sigma_{\pi, V}(x)\neq 0$. Then there is $n\in \mathbb N$ such that   $\sigma_{\pi_n, V_n}(x)\neq 0$.
\end{lemma}
\begin{proof} In the same way as in Lemma \ref{SOT-limit} one shows that $\sigma_{\pi, V}=\!SOT\!-\!\lim \sigma_{\pi_n, V_n}$ and the statement follows.
\end{proof}

We will need one more lemma, which is essentially contained in [\cite{DonRFD}, Lemma
1], where it is formulated in slightly different terms (see also [\cite{KristinTanya}, Lemma 2.5] for the proof).

\begin{lemma}[Hadwin \cite{DonRFD}]\label{DonUnitary} Let $U: H\to H$ be unitary, $H_n\subseteq H$ are subspaces such that $SOT\!-\!\lim P_{H_n} = 1_H$. Then there are  unitaries $U_n: H_n \to H_n$ such that $U = \ast\!-\!SOT\!-\!\lim U_n.$
\end{lemma}

\begin{theorem}\label{HNNRFD} Let $G$ be an amenable RF group and let $C$ be its finitely generated central subgroup such that there exists a $C$-filtration of $G$. Then  the HNN-extension $\langle G, t\;|\; t^{-1}Ct = C, \; id\rangle$ is RFD.
\end{theorem}
\begin{proof} Let $0\neq x\in C^*\langle G, t\;|\; t^{-1}Ct = C, \; id\rangle$. There is an irreducible representation $\pi$ of $C^*\langle G, t\;|\; t^{-1}Ct = C, \; id\rangle$ on a Hilbert space $K$ such that $\pi(x)\neq 0.$ Since $C^*(C)$ is a central $C^*$-subalgebra of $C^*\langle G, t\;|\; t^{-1}Ct = C, \; id\rangle$, there is a 1-dimensional representation $\lambda$ of $C^*(C)$ such that
$\pi(a) = \lambda(a) 1$, for any $a\in C^*(C)$.  Let
 $$\tilde\lambda (g)= \begin{cases} \lambda(g),\; g\in C \\ 0,\; g\notin  G\end{cases}.$$
 Let $\Lambda_{\tilde\lambda}$ be the GNS-representation of $C^*(G)$ with respect to the trace $\tilde\lambda$. The corresponding Hilbert spaces  will be denoted by $H$.

{\bf Claim:} There exists a unitary $U: K\to H^{(\infty)}$ such that
$$[U^* \Lambda_{\tilde\lambda}^{(\infty)}(c) U, \pi(t)]=0, $$ for any $c\in C$, and
$$\sigma_{U^* \Lambda_{\tilde\lambda}^{(\infty)} U, \pi(t)}(x)\neq 0.$$

{\it Proof of Claim:}
 By Corollary \ref{kernel} $\ker \Lambda_{\tilde\lambda} \subseteq \ker \pi|_{C^*(G)} $. Hence $$ rank \left(\pi|_{C^*(G)}\oplus \left(\Lambda_{\tilde\lambda}\right)^{\infty}\right)(a) = rank \left(\Lambda_{\tilde\lambda}\right)^{\infty}(a),$$ for any $a\in C^*(G)$. By Vociulescu's theorem the representations $\pi|_{C^*(G)}\oplus \left(\Lambda_{\tilde\lambda}\right)^{\infty}$ and $\left(\Lambda_{\tilde\lambda}\right)^{\infty}$ are approximately unitarily equivalent. On the other hand, by Corollary \ref{Don} the representation $\pi|_{C^*(G)}$ is the $\ast$-strong limit of unitary conjugates of $\pi|_{C^*(G)}\oplus \left(\Lambda_{\tilde\lambda}\right)^{\infty}$. All together this says us that $\pi|_{C^*(G)}$ is the pointwise $\ast$-strong limit of unitary conjugates of $\left(\Lambda_{\tilde\lambda}\right)^{\infty}$. We notice also that  by Lemma \ref{scalar} $\left(\Lambda_{\tilde\lambda}\right)^{(\infty)}(a) = \lambda(a)1$, for any $a\in C^*(C)$, and hence the same equality is true for unitary conjugates of $\left(\Lambda_{\tilde\lambda}\right)^{(\infty)}$. Hence the restriction of any unitary conjugate of $\Lambda_{\tilde\lambda}^{(\infty)}$ onto $C^*(C)$ commutes with $\pi(t)$. By Lemma \ref{HNNlimit} we conclude that there exists a unitary $U: K\to H^{(\infty)}$ such that
$$\sigma_{U^* \Lambda_{\tilde\lambda}^{(\infty)} U, \pi(t)}(x)\neq 0.$$
Claim is proved.

By Corollary \ref{CorollaryTracesInducedFromCenter}  there exist finite-dimensional representations $\pi_n $ of $G$ such that
\begin{equation}
\tilde\lambda(g) = \lim_{n\to \infty} tr \pi_n(g),
\end{equation}
for any $g\in G$, and  $\pi_n(c) $ is a scalar matrix, for any $c\in C$.  By Lemma \ref{ExelLoring} there exist coisometries $V_n: H \to H_{tr_{\pi_n}}$ such that for the representations $\rho_n = {V_n}^*\Lambda_{tr {\pi_n}} V_n$ we have
$$\Lambda_{\tilde\lambda}(a) = SOT\!-\!\lim \rho_n(a),$$ for any $a\in C^*(G)$.
By Lemma \ref{GNSfin-dim}, $\Lambda_{tr {\pi_n}}$,  and hence $\rho_n$, are finite-dimensional representations,  and by Lemma \ref{scalar}, $\Lambda_{tr{\pi_n}}(c)$,  and hence $\rho_n(c)$,   are scalar matrices. By Lemma \ref{infty} there are finite multiples of $\rho_n$, let us call them $\tilde \rho_n$, such that
\begin{equation}\label{uparrowHNN}
{\Lambda_{\tilde\lambda}}^{(\infty)}(a) = SOT\!-\!\lim \tilde\rho_n(a),
\end{equation}
 for any $a\in C^*(G)$, and we still have  that for each $c\in C$,  $\tilde\rho_n(c)$  are scalar matrices. Denote by $H_n\subset H^{(\infty)}$ the essential spaces of the representations $\tilde \rho_n$. It follows from (\ref{uparrowHNN}) that
 \begin{equation}\label{uparrow} SOT\!-\!\lim P_{H_n} =  1_{H^{(\infty)}}.\end{equation}
  We define  $K_n \subset K$  by $$K_n = U^{-1}(H_n),$$ where $U$ is the unitary form the claim,  and we define unitaries  $U_{n}: K_n\to H_n$ by $$U_n = U|_{K_n}.$$ Then $U = \ast\!-\!SOT\!-\!\lim U_n$ and  also
 $$U^*\Lambda_{\tilde\lambda}^{(\infty)}(a)U = SOT\!-\!\lim U_{n}^*\tilde\rho_n(a)U_{n},$$ for any $a\in C^*(G)$.  By (\ref{uparrow}) and Lemma \ref{DonUnitary} there are unitaries $W_n\in B(K_n)$ such that  $\pi(t) = \ast\!-\!SOT\!-\!\lim W_n$. We also notice that for each $c\in C$,  $U_{n}^*\tilde\rho_n(c)U_{n}$  are scalar matrices and hence commute with $W_n$.
By Claim and Lemma \ref{HNNlimit} there is $n\in \mathbb N$ such that the finite-dimensional representation $\sigma_{U_{n}^*\tilde\rho_nU_{n}, W_n}$ does not vanish on $x$.
\end{proof}

\begin{corollary}\label{HNNCorollaryMainCases} Let $G$ be a finitely generated amenable group and let $C$ be its finitely generated central subgroup. Then the following are equivalent:

(i) the HNN-extension $\langle G, t\;|\; t^{-1}Ct = C, \; id\rangle$ is RF;

(ii) there exists a $C$-filtration of $G$;

(iii) the HNN-extension $\langle G, t\;|\; t^{-1}Ct = C, \; id\rangle$ is RFD.
\end{corollary}
\begin{proof} (i) $\Rightarrow$ (ii): In  the case $C\neq G$ (i) $\Leftrightarrow$ (ii) by [\cite{Wong}, Th. 2.3]. The case $C=G$ is obvious. Indeed (i) implies that $G$ is RF, and then a $G$-filtration of $G$ obviously exists.

(ii) $\Rightarrow$ (iii):
This implication is proved in Theorem \ref{HNNRFD}.

(iii) $\Rightarrow$ (i): (iii) implies that $\langle G, t\;|\; t^{-1}Ct = C, \; id\rangle$ is MAP. Since it is finitely-generated, it is RF by Proposition \ref{MAPhenceRF}.
\end{proof}

\subsection{From HNN-extensions to amalgamated free products}

Here amalgams need not be central.

It is well-known that many results on residual finiteness of amalgamated free products of groups can be obtained from results on HNN-extensions by embedding an amalgamated free product of $A$ and $B$ into a certain HNN-extension of $A\ast B$ (see \cite{survey}). Below we show that for the RFD-property there is a different connection between amalgamated free products and HNN-extensions which arises by "purely operator-algebraic" reasons.


\begin{theorem}\label{FromHNNtoAmalgam} Let $A$ be a separable C*-algebra (or a group) and $B$ its C*-subalgebra (a subgroup, respectively). If the HNN-extension $\langle A, t\;|\; t^{-1}Bt = B, \; id\rangle$
is RFD, then the amalgamated free product $A\ast_{B=B} A$ is RFD.
\end{theorem}
\begin{proof} Let $\{\pi_{\alpha}\}$ be the set of all irreducible representations of $A\ast_{B=B} A$ up to unitary equivalence and let $\pi = \oplus \pi_{\alpha}$. Let $\rho$, $\gamma$, $\rho_{\alpha}$  and $\gamma_{\alpha}$ be the restrictions of $\pi$ and $\pi_{\alpha}$ to the first and second copies of $A$ respectively. In other words, $$\pi_{\alpha} = \sigma_{\rho_{\alpha},\; \gamma_{\alpha}},\; \pi = \sigma_{\rho,  \gamma}.$$ We notice that if for two representations $\tau_1, \tau_2$ of $A$ the representation $\sigma_{\tau_1, \tau_2}$ is irreducible, then so is the representation $\sigma_{\tau_2, \tau_1}$. This implies that each direct summand of $\rho = \oplus \rho_{\alpha}$ is also a direct summand of $\gamma = \oplus \gamma_{\alpha}$.  It follows that $\rho $ and $\gamma $ are unitarily equivalent and thus there is a unitary $u$ such that $$\rho(a) = u^*\gamma(a)u, $$ for any $a\in A$. Since $\rho(b) = \gamma(b),$ for any $b\in B$, we have $$[u, \gamma(b)]=0,$$ for any $b\in B$. Hence the pair $\gamma, u$ defines a representation of the HNN-extension $\langle A, t\;|\; t^{-1}Bt = B, \; id\rangle$ (on the Hilbert space $H$ of the representation $\pi$). Since $\langle A, t\;|\; t^{-1}Bt = B, \; id\rangle$ is RFD, by Exel-Loring [\cite{ExelLoring}, Th. 2.4]  there are representations $\tilde\gamma_n$ of $A$ living on finite-dimensional subspaces $H_n$ of $H$  and unitaries $u_n\in B(H_n)$ such that $u= \ast-SOT-\lim u_n$, $\gamma(a) = \ast-SOT-\lim \tilde\gamma_n(a)$, for each $a\in A$, and $[u_n, \tilde\gamma_n(b)]=0$, for any $b\in B$. Let $\tilde\rho_n = u_n^*\tilde\gamma_n u_n.$ Then $$\rho(a) = \ast-SOT-\lim \tilde\rho_n(a), $$
 for any $a\in A$ and
 $$\tilde\rho_n(b) = \tilde\gamma_n(b),$$
  for any $b\in B$. The latter equality implies that the pair $\tilde\rho_n, \tilde\gamma_n$ defines a representation of $A\ast_{B=B} A$. As $\pi$ is faithful, Lemma \ref{SOT-limit} shows that the representations $\sigma_{\tilde\rho_n, \tilde\gamma_n}$, $n\in \mathbb N$,  form a separating family of finite-dimensional representations of $A\ast_{B=B} A$.
  \end{proof}

\section{Negative results: a question of Khan and Morris and central amalgamation of RFD C*-algebras}

Trying to find a counterexample showing that central amalgams of RFD C*-algebras need not be RFD, we notice that Corollary \ref{CorollaryMainCases} indicates  natural candidates to test for such a counterexample. Namely we should look  at $C^*(G)\ast_{C^*(N)} C^*(G)$, where $G$ is an RF amenable group and $N$ is its finitely generated central subgroup such that $G/N$ is not RF.  Moreover, Corollary \ref{HNNCorollaryMainCases} (together with Proposition \ref{QuotientRF0}) asserts that in this case the corresponding HNN-extension is indeed not RFD.
We therefore will try to pass from HNN-extensions to amalgamated free products. This approach will work for the RF property as well.

\begin{lemma}\label{HNNnotRF} Let $\Gamma$ be a group and $N$ its central subgroup such that $\Gamma/N$ is not RF. Then the HNN-extension $\langle \Gamma, t\;|\; t^{-1}Nt = N, \; id\rangle$ is not RF.
\end{lemma}
\begin{proof} Since $\Gamma/N$ is not RF, it follows from Proposition \ref{QuotientRF0} that $\Gamma$ is not $N$-separable. Now by [\cite{Wong}, Th. 2.3] the HNN-extension $\langle \Gamma, t\;|\; t^{-1}Nt = N, \; id\rangle$ is not RF.
\end{proof}

\begin{lemma}\label{HNNembeds} Let $\Gamma$ be a group and $N$ its central subgroup. Then

 (i) if there exists $g_0\in \Gamma$ such that $g_0^k\notin N$, for any $k\in \mathbb Z$, then the  HNN-extension
  $\langle \Gamma, t\;| t^{-1}Nt = N, \; id\rangle$ embeds into the amalgamated free product $\Gamma \ast_{N=N} \Gamma$,

 (ii) if for any $K\in \mathbb Z$ there exists  $g_0\in \Gamma$ such that $g_0^k\notin N$, for any $-|K| \le k\le |K|$, and the HNN-extension $\langle \Gamma, t\;|\; t^{-1}Nt = N, \; id\rangle$ is not RF, then the amalgamated free product $\Gamma \ast_{N=N} \Gamma$ is not RF.
\end{lemma}
\begin{proof} Let $e\neq x\in \langle \Gamma, t\;|\; t^{-1}Nt = N, \; id\rangle$. One can reduce $x$ to one of the forms

$1) x = g_1t^{k_1}g_2t^{k_2}\ldots g_nt^{k_n},$ where $g_i \notin N, 0\neq k_i\in \mathbb Z, n\ge 2$,

$2) x = t^{k_1}g_2t^{k_2}\ldots g_nt^{k_n},$ where $g_i \notin N, 0\neq k_i\in \mathbb Z, n\ge 2$,

$3) x = g_1t^{k_1}g_2t^{k_2}\ldots g_n,$ where $g_i \notin N, 0\neq k_i\in \mathbb Z, n\ge 2$,

$4) x = t^{k_1}g_2t^{k_2}\ldots g_n,$ where $g_i \notin N, 0\neq k_i\in \mathbb Z, n\ge 2$,

$5) x = ht^{k},$ where $h\in N$, $0\neq k\in \mathbb Z$,

$6) x = h,$ where $h\in N$.

Let us denote the first (second, respectively)  copy of $\Gamma$ inside $\Gamma \ast_{N=N} \Gamma$  by $i_1(\Gamma)$ ($i_2(\Gamma)$, respectively). For any $g_0\in \Gamma$ we can define a homomorphism $$f: \langle \Gamma, t\;|\; t^{-1}Nt = N, \; id\rangle \to \Gamma \ast_{N=N} \Gamma$$ by $$f(t) = i_2(g_0), \; f(g) = i_1(g), \; g\in \Gamma.$$

Let now $g_0$ be as in the assumption (i) of the lemma. Then for any $e\neq x\in \langle \Gamma, t\;|\; t^{-1}Nt = N, \; id\rangle$ we have
$f(x)\neq e$. Indeed if $x$ is written in the form 5) or 6) then it is obvious, and if it is written in one the first 4 forms then it follows from the Normal Form Theorem for amalgamated free products. Thus $f$ is injective.

We now prove (ii). Since  $\langle \Gamma, t\;|\; t^{-1}Nt = N, \; id\rangle$ is not RF, there is $e\neq x\in \langle \Gamma, t\;|\; t^{-1}Nt = N, \; id\rangle$ that vanish under any homomorphism to a finite group. Let $K = \max |k_j|$ over all $k_j$'s appearing in the reduced form of $x$. By assumption we can choose $g_0\in \Gamma$ such that $g_0^k\notin N$, for any $-K \le k\le K$ and then as above we have $f(x)\neq e$.
 If $\Gamma \ast_{N=N} \Gamma$ was RF, there would be a homomorphism $\alpha$ from $\Gamma \ast_{N=N} \Gamma $ to a finite group such that $\alpha(f(x))\neq e$ and therefore $x$ would not vanish under the homomorphism $\alpha\circ f$, a contradiction.
\end{proof}

Recall the construction of Abels's group \cite{Abels}:
$$\Gamma = \left\{ \left(\begin{array}{cccc}1&x_{12}&x_{13}&x_{14}\\ 0 &p^k&x_{23}&x_{24}\\0&0&p^n&x_{34}\\0&0&0&1 \end{array}\right) : x_{ij}\in \mathbb Z\left[\frac{1}{p}\right] , k, n\in \mathbb Z\right\}.$$ Let
$$N = \left\{ \left(\begin{array}{cccc}1&0&0&x \\0&1&0&0\\ 0&0&1&0\\ 0&0&0&1 \end{array}\right) : x\in \mathbb Z \right\}.$$
It is easy to check that $N$ is a central subgroup of $\Gamma$. Abels showed in \cite{Abels} that $\Gamma$ and $\Gamma/N$ are finitely presented groups, $\Gamma$ is RF (and hence MAP), and $\Gamma/N$ is not RF.

\medskip

Now we can answer a question of Khan and Morris [\cite{KhanMorris}, p.428] of whether free product of MAP topological groups amalgamated over a common closed central subgroup is always MAP.

\begin{proposition}\label{Khan-Morris} Let $\Gamma$ and $N$  be as above. Then the amalgamated free product $\Gamma \ast_{N=N} \Gamma$ is not MAP.
\end{proposition}
\begin{proof} Since $\Gamma/N$ is not RF (\cite{Abels}), by Lemma \ref{HNNnotRF} the HNN-extension $\langle \Gamma, t\;|\; t^{-1}Nt = N, \; id\rangle$ is not RF. Let $g_0 = diag(1, p, p, 1)\in \Gamma$. Then $g_0^k\notin N$ for any $k\in \mathbb Z$. By Lemma \ref{HNNembeds} the amalgamated free product $\Gamma \ast_{N=N} \Gamma$ is not RF. Since $\Gamma$ is finitely generated, so is $\Gamma \ast_{N=N} \Gamma$. Hence  $\Gamma \ast_{N=N} \Gamma$ is not MAP by Proposition \ref{MAPhenceRF}.
\end{proof}

Since $\Gamma$ is solvable (\cite{Abels}) (hence amenable)  and MAP, the C*-algebra $C^*(\Gamma)$ is RFD by \cite{BekkaLouvet}.  Since $C^*\left(\Gamma \ast_{N=N} \Gamma \right) = C^*(\Gamma)\ast_{C^*(N)} C^*(\Gamma)$, we conclude

\begin{corollary} Central amalgamated free products of RFD C*-algebras need not be RFD.
\end{corollary}

\begin{remark}\label{RFnotenough} Proposition \ref{Khan-Morris}, together with the strategy of the proof of Theorem \ref{main}, shows that our refined approximation of characters as in Theorem \ref{TracesInducedFromCenter}  and Corollary \ref{CorollaryTracesInducedFromCenter} does not hold if one assumes only RF instead of the existence of (compatible) filtrations. \end{remark}

\end{document}